\documentclass[11pt]{amsart}

\usepackage{enumerate,url,amssymb,  nicefrac, mathrsfs, graphicx, pdfsync}
\newtheorem{theorem}{Theorem}[section]
\newtheorem{lemma}[theorem]{Lemma}
\newtheorem*{lemma*}{Lemma}

\newtheorem{proposition}[theorem]{Proposition}
\newtheorem{corollary}[theorem]{Corollary}

\theoremstyle{definition}
\newtheorem{definition}[theorem]{Definition}

\theoremstyle{remark}
\newtheorem{remark}[theorem]{Remark}

\numberwithin{equation}{section}

\newcommand{\abs}[1]{\lvert#1\rvert}

\newcommand{\norm}{\,|\!|\,}

\newcommand{\R}{\mathbb{R}}
\newcommand{\X}{\mathbb{X}}

\newcommand{\Y}{\mathbb{Y}}

\newcommand{\onto}{\xrightarrow[]{{}_{\!\!\textnormal{onto\,\,}\!\!}}}
\newcommand{\into}{\xrightarrow[]{{}_{\!\!\textnormal{into\,\,}\!\!}}}
\newcommand{\bydef}{\stackrel {\textnormal{def}}{=\!\!=} }

\DeclareMathOperator{\dist}{dist}

\def\le{\leqslant}
\def\ge{\geqslant}

\begin{document}

\title[Bi-conformal Energy and Quasiconformality]{Deformations of Bi-conformal Energy and\\ a new Characterization of Quasiconformality}
\author[T. Iwaniec]{ Tadeusz Iwaniec}
\address{Department of Mathematics, Syracuse University, Syracuse,
NY 13244, USA}
\email{tiwaniec@syr.edu}

\author[J. Onninen]{Jani Onninen}
\address{Department of Mathematics, Syracuse University, Syracuse,
NY 13244, USA and Department of Mathematics and Statistics, P.O.Box 35 (MaD) FI-40014 University of Jyv\"askyl\"a, Finland}
\email{jkonnine@syr.edu}

\author[Z. Zhu]{Zheng Zhu}
\address{Department of Mathematics and Statistics, P.O.Box 35 (MaD) FI-40014 University of Jyv\"askyl\"a, Finland}
\email{zheng.z.zhu@jyu.fi}

\thanks{T. Iwaniec was supported by the NSF grant DMS-1802107.
J. Onninen was supported by the NSF grant DMS-1700274. This research was done while Z. Zhu was visiting Mathematics Department at Syracuse University. He wishes to thank SU for the hospitality.}




\begin{abstract}
The concept of hyperelastic deformations of  bi-conformal energy is developed as an extension of quasiconformality.
These  are homeomorphisms $\, h \colon  \mathbb X \onto \mathbb Y\,$ between domains $\,\mathbb X, \mathbb Y \subset \mathbb R^n\,$
of the Sobolev class $\,\mathscr W^{1,n}_{\textnormal{loc}} (\mathbb X, \mathbb Y)\,$  whose inverse $\, f \bydef h^{-1} \colon  \mathbb Y \onto \mathbb X\,$
also belongs to $\,\mathscr W^{1,n}_{\textnormal{loc}}(\mathbb Y, \mathbb X)\,$. Thus the paper  opens  new topics in Geometric Function Theory (GFT)
with connections to mathematical models of Nonlinear Elasticity (NE). In seeking differences and similarities with quasiconformal mappings we examine closely the modulus of continuity of deformations of bi-conformal energy.  This leads us to a new characterization of quasiconformality. Specifically, it is observed that quasiconformal mappings behave locally at every point like radial stretchings. Without going into detail, if a quasiconformal map $\,h\,$  admits a function $\,\phi\,$ as its \textit{optimal modulus of continuity} at a point $\,x_\circ \, $, then  $\,f =  h^{-1}\,$ admits the inverse function $\, \psi = \phi^{-1}\,$ as its modulus of continuity at $\, y_\circ = h(x_\circ) \,$.  That is to say; a poor (possibly harmful)  continuity of $\,h\,$ at a given point $\,x_\circ\,$ is always compensated by a better continuity of $\,f\,$ at $\,y_\circ\,$, and vice versa. Such a gain/loss property,  seemingly overlooked by many authors,  is actually characteristic of quasiconformal mappings. 
 
 It turns out that the elastic deformations of bi-conformal energy are very different in this respect. Unexpectedly, such a map may have the same optimal modulus of continuity as its inverse  deformation.  In line with Hooke's Law,  when trying to restore the original shape of the body  (by the inverse transformation) the modulus of continuity may neither be improved nor become worse.
  However, examples to confirm this phenomenon are far from being obvious; indeed, elaborate computations are on the way. We eventually hope that our examples will gain an interest in the materials science, particularly in mathematical models of hyperelasticity.  
\end{abstract}

\maketitle

\section{Introduction} We study Sobolev homeomorphisms $\,h \colon \mathbb X \onto \mathbb Y\,$   between domains $\, \mathbb X , \mathbb Y  \subset \mathbb R^n\,$, together with their inverse mappings denoted by $\, f \bydef h^{-1} \colon \mathbb Y  \onto \mathbb X\,$. We impose two standing conditions on these mappings:

\begin{itemize}

\item The \textit{conformal energy} of $\,h\,$  (stored in $\, \mathbb X$)  is finite; that is,
\begin{equation}\label{eq:conf}
 \mathbf E_\mathbb X [h ] \bydef \int_\mathbb X \,| Dh(x)|^n \, \textnormal{d} x  \, < \infty 
 \end{equation}
 \item
 The conformal energy of $\,f\,$ (stored in $\,\mathbb Y$)  is also finite;
 \begin{equation}
 \mathbf E_\mathbb Y [f ] \bydef \int_\mathbb Y \,| Df(y)|^n \, \textnormal{d} y  \, < \infty 
 \end{equation}
\end{itemize}
Hereafter, $\abs{A}$ stands for the {\it Hilbert-Schmidt norm}
of a linear map $A$, defined by the rule $\abs{A}^2 = \textnormal{Tr}\, (A^t A)$.  
It should be noted that  the above energy integrals are invariant under conformal change of variables in their domains of definition ($\,\mathbb X\,$ and $\,\mathbb Y\,$, respectively). 
 This    motivates us calling such homeomorphisms
 \begin{center} \textit{\textbf{Deformations of Bi-conformal Energy}}
 \end{center}
Clearly, such deformations  include {\it quasiconformal} mappings. A Sobolev homeomorphisms $\,h \colon \mathbb X \onto \mathbb Y\,$ is said to be a quasiconformal mapping if there exists a constant ${\bf K}$ such that
\begin{equation}\label{eq:quasianalytic}
\abs{Dh (x)}^n \le {\bf K} J(x,h)\, ,  \qquad J(x,h)=\det Dh(x) \, .  
\end{equation}
The conformal energy integral~\eqref{eq:conf}, an  $n$-dimensional alternative to the classical Dirichlet integral, has drawn the attention of researchers in the  multidimensional GFT  \cite{BI, HKMb, HKb,   IMb, Reb,  Rib, Vub}. In Geometric Analysis the Sobolev space $\,\mathscr W^{1,n}(\mathbb X, \mathbb R^n) \,$  plays a special role for several reasons. First, this space  is on the edge of the continuity properties of Sobolev's mappings. Second, just the fact that $\,h\,$ is a homeomorphism  allows us to establish uniform bounds of its modulus of continuity. Precisely, given a compact subset $\,\mathbf X \Subset \mathbb X\,$, there exists a constant $\, C(\mathbf X, \mathbb X)\,$ so that for all distinct points $\; x_1 , x_2 \,\in  \mathbf X\,$, we have:

\begin{equation}\label{WellKnownEstimateX}
 |\,h(x_1)  - h(x_2) \,|  \; \leqslant \frac{ C(\mathbf X, \mathbb X)\;\sqrt[n]{\mathbf E_\mathbb X [h ]}}{\log^{\frac{1}{n}}\left( 1 \, +\, \frac{\textnormal{diam}\,\mathbf X }{ |x_1 - x_2 |} \right)}
\end{equation}
For a historical account and more details concerning this estimate we refer the reader to Section 7.4, Section 7.5  and Corollary 7.5.1 in the monograph \cite{IMb}.

For the same reasons, to every compact  $\,\mathbf Y \Subset \mathbb Y\,$ there corresponds a constant $\,C(\mathbf Y,\mathbb Y)\,$ such that for all distinct points $\; y_1 , y_2 \,\in  \mathbf Y\,$, we have:

\begin{equation}\label{WellKnownEstimateX}
 |\,f(y_1)  - f(y_2) \,|  \; \leqslant \frac{ C(\mathbf Y, \mathbb Y)\;\sqrt[n]{\mathbf E_\mathbb Y [f ]}}{\log^{\frac{1}{n}}\left( 1 \, +\, \frac{\textnormal{diam}\,\mathbf Y }{ |y_1 - y_2 |} \right)}
\end{equation}

 In other words, $\,h\,$ and $\,f\,$  admit the same function   $ \,\omega = \omega(t)   \approx \log^{-\frac{1}{n}} \left(1 +  1/t \right)\, $ as a  modulus of continuity.  Shortly, $\,h\,$ and $\,f\,$ are $\,\omega$-continuous.   There is still a slight  improvement to these estimates; namely,
 \begin{equation}
 \lim_{|x_1 - x_2| \rightarrow 0 }  |\,h(x_1)  - h(x_2)\,| \, \log^{\frac{1}{n}}\left( 1 \, +\, \frac{\textnormal{diam}\,\mathbf X }{ |x_1 - x_2 |} \right)\; =\; 0
 \end{equation}
 The question whether  the modulus of continuity  $ \,\omega = \omega(t)   \approx \log^{-\frac{1}{n}} \left(1 +  1/t \right)\, $ is the best and universal for all bi-conformal energy mappings remains unclear.  We shall not enter this issue here. The {\it optimal modulus of continuity} of $h \colon \X \onto \Y$ at a given point $x_\circ \in \X$  is defined  by
 \begin{equation}\label{eq:modcont}
 \omega_h (x_\circ ; t) \bydef \max_{\abs{x-x_\circ} =t}  \abs{h(x)-h(x_\circ)} \qquad \textnormal{for } 0\le t < \dist (x_\circ , \partial \X) \, . 
 \end{equation}
 Nevertheless, it is easy to see, via examples of radial stretchings, that in the class of functions that are powers of logarithms the exponent $\, \alpha = \frac{1}{n}\,$ is sharp; meaning that for $\, \alpha > \frac{1}{n}\,$ it is not generally true that
\begin{equation} \label{GeneralModulusOfContinuity}
 |\, h(x_1)  \,-\, h(x_2)\,| \; \preccurlyeq \; \log^{- \alpha } \left( 1 \, +\, \frac{\textnormal{diam}\,\mathbf X }{ |x_1 - x_2 |} \right)\;\; \footnote{Hereafter the notation $\,\mathbf A \preccurlyeq \mathbf B$ stands for the inequality $\,\mathbf A \leqslant  c\, \mathbf B$ in which $\, c>0\,$, called implied or hidden constant, plays no role. The implied constant may vary from line to line and is easily  identified from the context, or explicitely specified if necessary.}
 \end{equation}
To this end, we take a quick look at the radial homeomorphism $\,h \colon \mathbb B^n \onto \mathbb B ^n\,$ of the unit ball $\,\mathbb B^n \subset \mathbb R^n\,$ onto itself,
 \begin{equation}\label{RadialStretching}
  h(x) = \frac{x}{|x| \,\big( 1 - \log |x| \big)^{\frac{1}{n}}\; \big[ \,\log( e - \log |x|)   \,\big] ^\beta        }  \;\;,\;\;\textnormal{where}  \;\;\;\beta > \frac{1}{n}
 \end{equation}
It is often seen that the inverse map $\, f \bydef h^{-1} \colon  \mathbb Y \rightarrow \mathbb X\,$  admits better modulus of continuity than $\,h\,$, or  vice versa. Just for $\,h\,$ defined in (\ref{RadialStretching}), its inverse is  even $\,\mathscr C^\infty\,$-smooth.  Such a gain/loss rule about the moduli of continuity for a map and its inverse is typical of the radial stretching/squeezing. It turns out that  the  gain/loss rule gives a new characterization for a widely studied class of  quasiconformal mappings. 
\begin{theorem}\label{thm:qc}
Let $h \colon \X \onto \Y$ be a homeomorphism between domains $\X, \Y \subset \R^n$ and let $f \colon \Y \onto \X$ denote its inverse.  Then $h$ is quasiconformal if and only if for every pair $(x_\circ , y_\circ) \in \X \times \Y$, $y_\circ =h(x_\circ)$, the optimal modulus of continuity functions $\omega_h = \omega_h(x_\circ ; t)$ and $\omega_f = \omega_f(y_\circ ; s)$ are quasi-inverse to each other; that is, there is a constant $\mathcal K \ge 1$ (independent of $(x_\circ , y_\circ)$) such that
\[\mathcal K^{-1} s \le (\omega_h \circ \omega_f) (s) \le \mathcal K s\]
for sufficiently small $s>0$.
\end{theorem}
See Section~\ref{sec:qcmaps} for fuller discussion. It should be noted that for a radial stretching/squeezing homeomorphism $h(x)= {\bf H}(\abs{x}) \frac{x}{\abs{x}}$, ${\bf H}(0)=0$,  we always have
\[ (\omega_h \circ \omega_f) (s) \equiv s\]
Thus it amounts to saying that 
\begin{center}
{\it Quasiconformal mappings are characterized by being comparatively radial strectching/squeezing at every point.}
\end{center}

  At the first glance, the gain/loss rule seems to generalize to deformations of bi-conformal energy. Here  we refute this view, by constructing examples in which  both $\,h\,$ and $\,f\,$  admit the same   modulus of continuity. These examples work well regardless of whether or not the modulus of continuity (given upfront)  is close to the  borderline case $ \,\omega = \omega(t)   \approx \log^{-\frac{1}{n}} \left(1 +  1/t \right)\, $.  Without additional preliminaries, we now can  illustrate this instance with a representative case of Theorem \ref{MainTheorem}.

\begin{theorem} [A Representative Example]\label{BasicExampleTheorem} Consider a modulus of continuity function  $\,\phi \colon  [0, \infty)  \onto  [0, \infty)\,$ defined by the rule
 \begin{equation} \label{BasicExample}
 \phi(s) = \left\{\begin{array}{ll}
  0 & \textnormal{if} \;s = 0  \\
  \left[\, \log  \left( \frac{e}{s} \right)\right]^{-\frac{1}{n}} \; \left [\,\log \log \left(\frac{e^e}{s}  \right)\;\right]^{-1}& \textnormal{if}\;\; 0 < s \leqslant 1\\
  s & \textnormal{if} \; s \geqslant 1\end{array} \right.
 \end{equation}
 Then there exists a deformation of  bi-conformal energy  $\,H \colon \mathbb R^n \onto \mathbb R^n\,$ such that
 \begin{itemize}
 \item  $\,H(0) = 0\;, \;\; H(x) \equiv x \,,\;\textnormal{for}\; |x| \geqslant 1 \,$
 \item  $ |\,H(x_1) \,-\, H(x_2)\,|  \; \, \preccurlyeq\,  \; \phi(|x_1 - x_2|) $\; \;, \; for all $\,x_1 , x_2  \in \mathbb R^n\,$
 \end{itemize}
Its inverse  $\,F \bydef H^{-1}  \colon  \mathbb R^n \onto \mathbb R^n\,$ also admits $\,\phi\,$ as a modulus of continuity,
 \begin{itemize}
 \item  $ |\,F(y_1) \,-\, F(y_2)\,|  \; \, \preccurlyeq\,  \; \phi(|y_1 - y_2|) $\; \;,\;\; for all \,$\,y_1, y_2 \in \mathbb R^n\,$ \;\;\;\footnote{ \,In the above estimates the implied constants depend only on $\,n\,$.}
 \end{itemize}
 Furthermore, $\,\phi\,$ represents  the optimal  modulus of continuity at the origin for  both $\,H\,$ and $\,F\,$; that is,   for every $\,0 \leqslant s < \infty\,$ we have
 \begin{equation}\label{OptimalModulus1}
   \omega_H(0, s) \;  = \;   \phi(s)  =   \omega_F(0, s) \, . 
 \end{equation}
 \end{theorem}
 \begin{remark} More specifically, letting $\, \psi : [0, \infty) \onto [0,\infty) \,$  denote the inverse of $\,\phi\,$,   the maxima in (\ref{OptimalModulus1}) are attained on the vertical axes, where we have
 \begin{equation}\label{OptimalOscillationH}
 H(0,...,0,\; x_n)  =  \left\{\begin{array}{ll}
  (0, ... , 0,  \;\phi(x_n)\,)  & \textnormal{if} \;\; x_n \geqslant 0  \\
  (0, ... , 0 , \;\psi (x_n)\,) &  \textnormal{if}\; \; x_n \leqslant 0\end{array} \right.
 \end{equation}
\begin{equation}\label{OptimalOscillationF}
 F(0,...,0,\; y_n)  =  \left\{\begin{array}{ll}
  (0, ... , 0, \;\psi(y_n)\,)  & \textnormal{if} \;\; y_n \geqslant 0  \\
  (0, ... , 0 , \;\phi(y_n)\, ) &  \textnormal{if} \;\; y_n \leqslant 0\end{array} \right.
 \end{equation}
It is worth noting  here that in our representative examples the inverse function $\, \psi :  [ 0, \infty)  \onto [0, \infty)\,$ will be even $\,\mathcal C^\infty\,$-smooth near $\,0\,$.
\end{remark}
There are many more reasons for studying  deformations  of bi-conformal energy.  First, a homeomorphism $\,h\colon  \mathbb X \rightarrow \mathbb Y\,$ in $\,\mathscr W^{1,n}(\mathbb X, \mathbb Y) \,$ whose inverse $\, f \bydef h^{-1} \colon  \mathbb Y \rightarrow \mathbb X\,$ also lies in  $\,\mathscr W^{1,n}(\mathbb Y, \mathbb X) \,$ include ones with  \textit{integrable inner distortion}, see~\eqref{eq:totalenergy2}.   From this point of view our study not only expands  the theory of quasiconformal mappings but also  mappings of finite distortion. The latter can be traced back to the early paper by Goldstein and Vodop'yanov  \cite{GoldVod} (1976)\, who established  continuity of such mappings. However, a systematic study of mappings of finite distortion has begun in 1993 with planar mappings of integrable distortion \cite{IwaniecSverak} (Stoilow factorization), see also the monographs \cite{AIMb, IMb, HKb}. The optimal modulus of continuity for mappings of finite distortion and their inverse deformations have been studied in numerous publications~\cite{CaH, ClH, HeK, Hi, IwKosOnn1, KoO, OnT, OnZ}. In all of these results, except in~\cite{OnT}, the sharp modulus of continuity is obtained among the class of radially symmetric mapping. 

In a different direction, the essence of elasticity is reversibility. All materials have limits  of the admissible distortions. Exceeding such a limit one breaks the internal structure of the material (permanent damage).  Here we  take on stage the materials of \textit{bi-conformal stored-energy}
\begin{equation}\label{TotalEnergy1}
\mathbf E_{\mathbb X \mathbb Y} [ h,f]  \,\bydef \, \mathbf E_\mathbb X [h ] +\,\mathbf E_\mathbb Y [f ] = \int_\mathbb X \,| Dh(x)|^n \textnormal{d} x  \;+\;   \int_\mathbb Y \,| Df(y)|^n \textnormal{d} y
\end{equation}
The bi-conformal energy  reduces to an integral functional defined solely over the domain $\,\mathbb X\,$ by the rule:
\begin{equation}\label{eq:totalenergy2}
\mathbf E_{\mathbb X \mathbb Y} [ h,f]  \,  = \mathscr E_\mathbb X[h] \,\bydef\,   \int_\mathbb X \,\Big\{\,\big|\, Dh(x)\,\big|^n \, +\, \frac{\big|\,D^\sharp h(x)\,\big|^n}{ \; \left [\,\mathbf J_h(x)\,\right]^{n-1} }\, \Big\}\,\; \textnormal{d} x
\end{equation}
where the ratio term represents the {inner distortion} of $\,h\,$. For more details we refer the reader to \cite{AIMO}.  Examples abound in which one can return the deformed body to its original shape with conformal energy, but not necessarily via the inverse mapping $\,f = h^{-1}  \colon \mathbb Y \onto \mathbb X\,$, because $\,f\,$ need not even belong to $\,\mathscr W^{1,n}(\mathbb Y, \mathbb R^n) \,$.  This typically occurs when the boundary of the deformed configuration (like a ball with a straight line slit cut)  differs topologically from the boundary of the reference configuration (like a ball without a cut)~\cite{IOnn1, IOhyper, IOnn2}.   We believe that the geometric/topological obstructions for  reversibility of elastic deformations might be of interest in mathematical models of nonlinear elasticity (NE) \cite{Anb, Bac, Cib, MHb}. In our setting, by virtue of the Hooke's Law, it is  naturally to study deformations of bi-conformal energy.  One of the  important  problems in nonlinear elasticity is whether or not a  radially symmetric solution of a rotationally invariant minimization problem is indeed the absolute minimizer. In the case of bi-conformal energy this is proven to be the case in low dimension models  ($n=2,3$)~\cite{IOhy}. The radial symmetric solutions, however, may fail to be absolute minimizers if $n \ge 4$~\cite{IOhy}. Several more papers, in the intersection of NE and GFT, are devoted to understand the expected radial symmetric properties~\cite{AIM, Ba1, CG, HLW, Ho, IKOni, IKO3, IOne, IwOnMemo, JM, JK, KO, Me, MS, S, SiSp, St}.

\section{Quick review of the modulus of continuity}
Let us recall the concept of \textit{modulus of continuity}, also known as \textit{modulus of oscillation}; the concept introduced by H. Lebesgue~\cite{Lebesgue} in 1909.\\
We are dealing with continuous mappings $\,h : \mathbb X \rightarrow \mathbb Y\,$ between subsets $\,\mathbb X \subset \mathscr X\,$ and $\,\mathbb Y \subset \mathscr Y \,$ of normed spaces $\,(\mathscr X ,  | \cdot |)\,$ and $\,(\mathscr Y , \norm \cdot \norm)\,$.

A modulus of continuity  is any continuous function $\, \omega : [ 0 , \infty ) \rightarrow [0, \infty)\,$ that is strictly increasing and $\, \omega(0) = 0\,$.
\begin{definition}\label{LocalOscillations}

A  continuous mapping $\,h : \mathbb X \rightarrow \mathbb Y\,$  is said to admit $\,\omega\,$ as its (local) modulus of continuity at the point $\,x_\circ \in \mathbb X\,$ if
\begin{equation}\label{OmegaContiunuity}
\norm h(x) - h(x_\circ) \norm \; \preccurlyeq \;\omega(| x - x_\circ|)\;\,, \;\; \textnormal{for all} \,\, x \in \mathbb X
\end{equation}
Here the implied constant may depend on $\,x_\circ\,$, but not on $\,x\,$.
In short,  $\,h\,$ is $\,\omega\,$-continuous at the point $\,x_\circ\,$. If this inequality holds for all $\, x, x_\circ \in \mathbb X\,$ with an  implied constant independent of $\,x\,$ and $\,x_\circ\,$  then  $\,h\,$ is said to admit $\,\omega\,$ as its (global) modulus of continuity in $\,\mathbb X\,$.
\end{definition}

 \begin{definition}[Optimal Modulus of Continuity]  Every uniformly continuous function $\,h: \mathbb X \rightarrow \,\mathbb Y\,$ admits the  optimal  modulus of continuity at a given point $\, x_\circ \in \mathbb X\,$, given by the rule:
 \begin{equation}\label{OptimalModulusOfContinuity}
 \omega_{_h}(x_\circ ; t) \, \bydef \; \sup \{\norm\,h(x) - h(x_\circ) \,\norm \colon  x \in \mathbb X\;,\; |x - x_\circ | \leqslant t  \}
 \end{equation}
 No implied constant is involved in this definition. Similarly, the function
 \begin{equation}\label{GlobalOptimalModulusOfContinuity}
 \Omega_{_h}(t) \, \bydef \; \sup \{\norm\,h(x) - h(x_\circ) \,\norm \colon x, x_\circ \in \mathbb X\;,\; |x - x_\circ | \leqslant t  \}
 \end{equation}
is referred to as \textit{ (globally) optimal modulus of continuity} of $\,h\,$ in  $\,\mathbb X\,$.
 \end{definition}

\begin{definition}[Bi-modulus of Continuity] The term \textit{bi-modulus of continuity} of a homeomorphism $\,h : \mathbb X \onto \mathbb Y\,$ refers to a pair $ (\phi, \psi)\,$  of continuously increasing functions $\,\phi : [0, \infty) \onto [0, \infty)\,$  and $\,\psi : [0, \infty) \onto [0, \infty)\,$ in which $\,\phi\,$ is a modulus of continuity of $\,h\,$ and $\,\psi\,$ is a modulus of continuity of the inverse map $\, f \bydef h^{-1} : \mathbb Y \onto \mathbb X\,$.  Such a pair is said to be the optimal bi-modulus of continuity at the point $\,(x_\circ,y_\circ)  \in \mathbb X \times \mathbb Y\,$, $\, y_\circ = h(x_\circ)\,$,  if  $\,\phi(t) = \omega_{_h}(x_\circ ;t)\,$ and $\,\psi(s) = \omega_{_f}(y_\circ ;s)\,$

\end{definition}

\section{Quasiconformal Mappings}\label{sec:qcmaps}

 Let us take a quick look at the radial stretching/squeezing  homeomorphism \, $\,h : \mathbb R^n \onto \mathbb R^n\,$ defined by:
 \begin{equation}
  h(x) =  {\bf H}(|x|)\,\frac{x}{|x|}\;, \; \; \textnormal{for}\; x \in \mathbb R^n
 \end{equation}
where the function $\,{\bf H}: [0, \infty) \onto [0, \infty)\,$ (interpreted as radial stress function) is continuous and strictly increasing.  Its inverse $\,f \bydef h^{-1} : \mathbb R^n \onto \mathbb R^n\,$  becomes a squeezing/stretching homeomorphism of the form:

\begin{equation}
 f(y) =  {\bf F}(|y|)\,\frac{y}{|y|}\;, \; \; \textnormal{for}\; y \in \mathbb R^n
 \end{equation}
where $\,{\bf F}: [0, \infty) \onto [0, \infty)\,$ stands for the inverse function of $\,{\bf H}\,$. These two radial stress functions are exactly the optimal moduli of continuity at $\,0 \in \mathbb R^n\,$ of $\,h\,$ and $\,f\,$, respectively. By the definition,
\[
\begin{split}
\omega_h(t) &  \bydef \omega_h (0, t) = \max_{|x| = t}\,|h(x)|\; =\; {\bf H}(t) \\ 
\omega_f(s) &  \bydef \omega_f (0, s) = \max_{|y| = s}\,|f(y)|\; =\; {\bf F}(s) \, . 
\end{split}
\]
Therefore
\begin{equation}
\omega_f (\omega_h(t) )    \; \equiv t\,  \;\textnormal{for all}\;\; t \geqslant 0\;,\;\;\;\;  \textnormal{and}\; \;  \omega_h (\omega_f(s))  \equiv s\,\;\;\;\textnormal{for all}\;\;\;s \geqslant 0.
\end{equation}
The above identities admit of a simple interpretation:
\begin{center}\textsl{The better is the optimal modulus of continuity of $\, h\,$, the worse is \\the optimal modulus of continuity of its inverse map $\, f ,$ and vice versa.  }
\end{center}
Look at  the power type stretching $\,h(x) = |x|^N \frac{x}{|x|}\,$  and  $\,f(y) = |y|^{\frac{1}{N}} \frac{y}{|y|}\,. $ \\
To an extent, this interpretation pertains to all quasiconformal homeomorphisms. There are three main equivalent definitions for quasiconformal mappings: metric, geometric, and analytic. The {\it analytic definition}~\eqref{eq:quasianalytic} was first considered by Lavrentiev in connection with elliptic systems of partial differential equations. Here we will relay on  the {\it metric definition}, which says that ``infinitesimal balls are transformed to infinitesimal ellipsoids of bounded eccentricity.'' The interested reader is referred to~\cite[Chapter 3.]{AIMb} to find more about the foundations of quasicoformal mappings.
\begin{definition}\label{DefKquasiConf}

 Let $\,\mathbb X\,$ and $\,\mathbb Y\,$ be domains in $\,\mathbb R^n\,, \, n\geqslant 2\,$, and $\,h :  \mathbb X  \onto \mathbb Y\,$ a homeomorphism. For every point $\,x_\circ \in \mathbb X\,$ we define.

 \begin{equation}\label{TheHratio}
 \mathcal H_h(x_\circ ,  r ) \;\bydef\; \frac{ \max_{|x - x_\circ| = r} \;|h(x) - h(x_\circ)|}{\min_{|x - x_\circ| = r}\; |h(x) - h(x_\circ)|}
 \end{equation}
 whenever $\,0 < r < \textnormal{dist} (x_\circ, \partial \mathbb X)\,$.   Also define
 \begin{equation}
  1 \leqslant \mathcal H_h(x_\circ) \;  \bydef \; \limsup_{r\rightarrow 0\,} \mathcal H_h(x_\circ ,  r )  \leqslant \infty
 \end{equation}
 and call it the \textit{linear dilatation} of $\,h\,$ at $\,x_\circ\,$. If, furthermore,
 \begin{equation}
  \mathcal K_h \bydef \sup_{x_\circ \in \mathbb X} \mathcal H_h(x_\circ)\, < \infty
 \end{equation}
 then we call $\,\mathcal K_h\,$ the \textit{maximal linear dilatation} of $\,h\,$ in $\,\mathbb X\,$ and $\,h\,$ a quasiconformal mapping. Finally, $\,h\,$ is $\,K\,$-quasiconformal, $\,1\leqslant K < \infty\,$  if
\begin{equation}
 \textnormal{ess-}\!\!\sup _{x_\circ \in \mathbb X} \mathcal H_h(x_\circ)\, \leqslant K
 \end{equation}
 \end{definition}
It should be noted that the inverse map  $\, f \bydef h^{-1}  : \mathbb Y  \onto \mathbb X\,$ is also $\,K\,$-quasiconformal.

Next, we invoke the optimal modulus of continuity at a point $\,x_\circ \in \mathbb X\,:$
$$
\omega_h(t)  \bydef  \omega_h(x_\circ; t) \, = \,\max_{|x - x_\circ| = t}|h(x) - h(x_\circ)|\,,\, \textnormal{for}\, \,0 \leqslant t < t_\circ  \bydef \textnormal{dist} (x_\circ ;  \partial \mathbb X)\,.
$$
 This defines a continuous strictly increasing function $\,\omega_h  :  [0, t_\circ)  \onto [0, s_\circ)\,$, where $\, s_\circ \bydef \textnormal{dist} (y_\circ ; \partial \mathbb Y).\,$
Similar definitions apply to the inverse map $\, f : \mathbb Y \onto \mathbb X\,$ which is also $\,K\,$-quasiconformal. Its optimal modulus of continuity at the image point $\,y_\circ =  h(x_\circ)  \,$ is given by
$$
\omega_f(s) \bydef  \omega_f(y_\circ; s) \, = \,\max_{|y - y_\circ| = s}|f(y) - f(y_\circ)|\,,\;\; \textnormal{for}\; 0 \leqslant s  < s_\circ
$$
Therefore, both compositions $\, \omega_f (\omega_h(t)) \,$  and $\, \omega_h (\omega_f(s))\,$ are well defined for  $\,0 \leqslant t < t_\circ\,$ and $\,   0 \leqslant s  < s_\circ\,$, respectively.
Unlike the radial stretchings,  the function $\,\omega_f(s)\,$ is generally not the
inverse of $\,\omega_h(t)\,$, but very close to it. Namely, the optimal modulus of continuity
of $\,h\,$ and that of $\,f\,$ are \textit{quasi-inverse} to each other. Let us make this statement
more precise by the following theorem.
\begin{theorem} [Local quasi-inversion] \label{LocalQuasiInversion}Let a map $\,h : \mathbb X \onto \mathbb Y\,$ be $\,K\,$-quasiconformal and $\,f : \mathbb Y \onto \mathbb X\,$ denote its inverse. Then there is a constant $\, \mathscr K  = \mathscr K(n,K)\,\geqslant 1\,$ such that for every point $\,x_\circ \in \X \,$ and its image $\,y_\circ = h(x_\circ) \in \Y\,$ it holds
\begin{equation} \label{QuasiInverseBounds}
 \mathscr K^{-1} s \; \leqslant\,  \omega_h (\omega_f(s))  \;\leqslant\, \mathscr K  s\;\;\;\; \textnormal{and}\;\;\;\;  \mathscr K^{-1} t \;\leqslant\,  \omega_f (\omega_h(t))  \;\leqslant\;\mathscr K  t
\end{equation}
 whenever $\,0\leqslant t \leqslant t(x_\circ)\,$ and   $\, 0 \leqslant s \leqslant s(y_\circ)\,$. Here the upper bounds  positive numbers $t(x_\circ)$ and $s(y_\circ)$,   depend only on $\,\textnormal{dist}(x_\circ; \partial \mathbb X)\,$ and $\,\textnormal{dist} (y_\circ; \partial \mathbb Y)\,$, respectively.
\end{theorem}

Before proceeding to the proof, we recall a very useful Extension Theorem by F. W. Gehring  \cite{GerExtension},   see also the book by J. V\"{a}is\"{a}l\"{a} \cite{Vab} (Theorem 41.6). This theorem allows us to reduce a local quasiconformal problem to an analogous problem for mappings defined in the entire space $\,\mathbb R^n\,$.

\begin{lemma} [F. W.  Gehring] Every  quasiconformal map $\,h : \mathbb B(x_\circ, 2 r)  \into \mathbb R^n\,$ defined in a ball $\,\mathbb B(x_\circ, 2 r) \subset \mathbb R^n\,$ admits a quasiconformal mapping $\,h' :  \mathbb R^n \onto \mathbb R^n\,$ which equals $\,h\,$ on $\,\mathbb B(x_\circ,  r)\,$. The dilatation of $h'$  depends only  that of $h$ and the dimension  $\,n\,$.
\end{lemma}

 Accordingly, we may (and do) assume that $\,\mathbb X = \mathbb Y = \mathbb R^n\,$. This will give us a more precise information about the constant $\,\mathscr K = \mathscr K(n,K)\,$.

 \begin{theorem} [Global quasi-inversion] \label{GlobalInversion} Let a map $\,h : \mathbb R^n \onto \mathbb R^n\,$ be $\,K\,$-quasiconformal and $\,f : \mathbb R^n \onto \mathbb R^n\,$ denote its inverse. Then there is a constant $\, \mathscr K  = \mathscr K(n,K)\,\geqslant 1\,$ such that for every point $\,x_\circ \in \mathbb R^n\,$ and its image $\,y_\circ = h(x_\circ)\,$ it holds
\begin{equation} \label{GlobalQuasiInverseBounds}
 \mathscr K^{-1} s \; \leqslant\,  \omega_h (\omega_f(s))  \;\leqslant\, \mathscr K  s\;\; \;\;\textnormal{and}\;\;\;\;  \mathscr K^{-1} t \;\leqslant\,  \omega_f (\omega_h(t))  \;\leqslant\;\mathscr K  t
\end{equation}
for all $\, s \geqslant 0\,$ and  $\,t \geqslant 0\,$.
\end{theorem}

  Rather than using the original definition  we will appeal to Gehring's  characterization of quasiconformal
  mappings, see Inequality (3.3) in \cite{GehringMartio} and some related articles \cite{Ge, Gehring, Vub, Vab, Kelingos, Vaisala}. The interested reader is referred to a book by P. Caraman \cite{Caraman} on various definitions and extensive early literature on the subject.

 \begin{proposition}[Three points condition] \label{ThrePointsRatio}
 To every $\,\lambda \geqslant 1\,$ there corresponds a constant  $\,1 \leqslant \mathscr K_\lambda = \mathscr K_\lambda(n,K)\,$
 such that:\\
 Whenever three distinct points $\,x_\circ \,,\, x_1\,,\, x_2 \,\in\,\mathbb R^n\,$ satisfy the ratio condition
 \begin{equation}
  \frac{|x_1 - x_\circ|}{|x_2 - x_\circ|}\; \leqslant \lambda ,
 \end{equation}
the image points under $\,h : \mathbb R^n \onto \mathbb R^n\,$ satisfy analogous condition

\begin{equation} \label{Klambda}
  \frac{|h(x_1) - h(x_\circ)|}{|h(x_2) - h(x_\circ)|}\; \leqslant \;\mathscr K_\lambda = \mathscr K_\lambda(n,K)
 \end{equation}
\end{proposition}
In particular,

\begin{proposition}\label{MaxMin}
Let $\,h :  \mathbb R^n  \onto \mathbb R^n\,$ be $\,K\,$-quasiconformal. Then for every point $\,x_\circ \in \mathbb X\,$  and $\,0< r < \infty\,$ we have
 \begin{equation}\label{TheHratio}
 \mathcal H_h(x_\circ ,  r ) \;\bydef\; \frac{ \max_{|x - x_\circ| = r} \;|h(x) - h(x_\circ)|}{\min_{|x - x_\circ| = r}\; |h(x) - h(x_\circ)|}  \,\leqslant \mathscr K  = \mathscr K_1(n,K)
 \end{equation}
\end{proposition}

\begin{proof} (of Theorem \ref{GlobalInversion})\, It is clearly sufficient to make the computation
 when $\,x_\circ = 0\,$ and $\,y_\circ = 0\,$.  In this case the condition \,(\ref{TheHratio})\, takes the form

 \begin{equation} \label{DistortionOnSpheres}
 \;\frac{1}{\mathscr K} \,|h(x_2)| \leqslant  \,  |h(x_1)|\;  \leqslant \;\mathscr K \, |h(x_2)|\;,\;\;
 \textnormal{whenever}\, \,|x_1| = |x_2| \not = 0
 \end{equation}

 By the definition of the optimal modulus of continuity at the origin, we have:
\begin{itemize}
\item   $\omega_h (\omega_f(s)) \,=  |h(x)|\;$ for some $\, x \in \mathbb R^n\,$  with $\, |x|  = \omega _f(s)$
\item $ \omega _f(s) =  |f(y)|\,$  for some $\, y \in \mathbb R^n\,$ with $\, |y| = s\,$
\item  Therefore,   $\omega_h (\omega_f(s) )\,=  |h(x)|\;$, for some $\, |x|  = |f(y)|$
\end{itemize}
Now, the right hand side of inequality at \,(\ref{DistortionOnSpheres}) gives the desired
upper bound $\,\omega_h (\omega_f(s))\,= |h(x)| \leqslant \mathscr K \, |h(f(y))| = \mathscr K\, |y| = \mathscr K s\,$, whereas the left hand side gives the
lower bound $\,\omega_h (\omega_f(s))\,= |h(x)| \geqslant \mathscr K^{-1} \, |h(f(y))| = \mathscr K^{-1}\, |y| = \mathscr K^{-1}\,s\,$.
The analogous bounds for $\,\omega_f(\omega_h(t))\,$ at (\ref{QuasiInverseBounds}) \, follow by interchanging the roles of $\,h\,$ and $\, f\,$; as they are both $\,K\,$-quasiconformal. This completes the proof of Theorem \ref{GlobalInversion}.
\end{proof}
The converse statement to Theorem \ref{LocalQuasiInversion}  reads as:
\begin{theorem} \label{ConverseStatement} Consider a homeomorphism $\,h : \mathbb X \onto \mathbb Y\,$,  its inverse mapping $\,f : \mathbb Y \onto \mathbb X\,$,  and their optimal moduli of continuity at a point $\,x_\circ \in \mathbb X\,$ and $\,y_\circ = h(x_\circ)\,$, respectively:
$$
\omega_h (t)  \bydef \max_{|x - x_\circ| = t} |h(x) - h(x_\circ) |\;\;\;\textnormal{and} \;\;\;\;\; \omega_f (s) \bydef \max_{|y - y_\circ| = s} |f(y) - f(y_\circ) |
$$
for $\, 0 \leqslant t < \textnormal{dist}( x_\circ , \partial \mathbb X)\,$  and  $\, 0 \leqslant s < \textnormal{dist}( y_\circ , \partial \mathbb Y)\,.$ Assume the following \textit{one-sided quasi-inverse condition}  at every point $\,x_\circ \in \mathbb X\,$, with a constant $\, \mathscr K \geqslant 1\,.$
\begin{equation} \label{OneSidedCondition}
\omega_h (\omega_f(r)) \,\leqslant \mathscr K r \;\; \textnormal{for all sufficiently small}\; r > 0\,\;\;\;  (\textnormal{depending on} \,x_\circ )
\end{equation}
Then $\,h\,$ is $\,\mathscr K\,$-quasiconformal.
\end{theorem}
Here is a simple geometric proof.
\begin{proof}
We shall actually show that Condition  (\ref{OneSidedCondition})  at the given point $\,x_\circ \in \mathbb X\,$ implies

\begin{equation}\label{MaxMinRatio}
 \mathcal H_h(x_\circ ,  t ) =  \frac{ \max_{|x - x_\circ| = t} \;|h(x) - h(x_\circ)|}{\min_{|x - x_\circ| = t}\; |h(x) - h(x_\circ)|}  \,\leqslant \mathscr K\, ,   \end{equation}
 for $t>0$ sufficiently small.
 In particular, for every $\,x_\circ \in \mathbb X\,$ it holds that:
  \begin{equation}\label{MaxDistortion}
 \limsup_{t\rightarrow 0} \;\mathcal H_h(x_\circ ,  t )  \,\leqslant \mathscr K\;,\;\;\textnormal{as required. }
 \end{equation}
A sufficient  upper bound of $\,t\,$ at (\ref{MaxMinRatio})  depends on $\,\textnormal{dist}(x_\circ, \partial \mathbb X)\,$, but we shall not enter into this issue. It simplifies the writing, and causes no loss of generality, to assume that $\,x_\circ = y_\circ = 0\,$. Thus we are reduced to showing that
 \begin{equation}
 \max_{|x| = t} |h(x)|  \leqslant \mathscr K \min_{|x| = t} |h(x)|\;,\;\;\;\textnormal{for all sufficiently small} \,\; t > 0.
 \end{equation}
 To this end, consider  the ball $\,\mathbb B(x_\circ , t ) \subset \mathbb X\,$  centered at $\,x_\circ = 0\,$ and with small radius $\,t > 0\,$. Its image under $\,h\,$, denoted by $\,\Omega = h (\mathbb B(x_\circ, t ) ) \subset \mathbb Y\,$,  contains the origin $\,y_\circ = 0\,$.  Let $\,r > 0\,$ denote the largest radius of a ball, denoted by $\,\mathbf B_r \subset \Omega\,$,  centered at $\,y_\circ = 0\,$. Thus
 $$
 \min_{|x| = t} |h(x)|  =  r
 $$
Similarly,  denote by $\,R\,$ the smallest radius of a ball  $\,\mathbf B_R \supset \Omega\,$ centered at $\,y_\circ  = 0\,$, see Figure~\ref{TheRatioRandr}.  Thus
 $$
  R = \max_{|x| = t} |h(x)|  \bydef  \omega_h(t)
 $$

  \begin{figure}[!h]
  \begin{center}
\includegraphics*[height=2.0in]{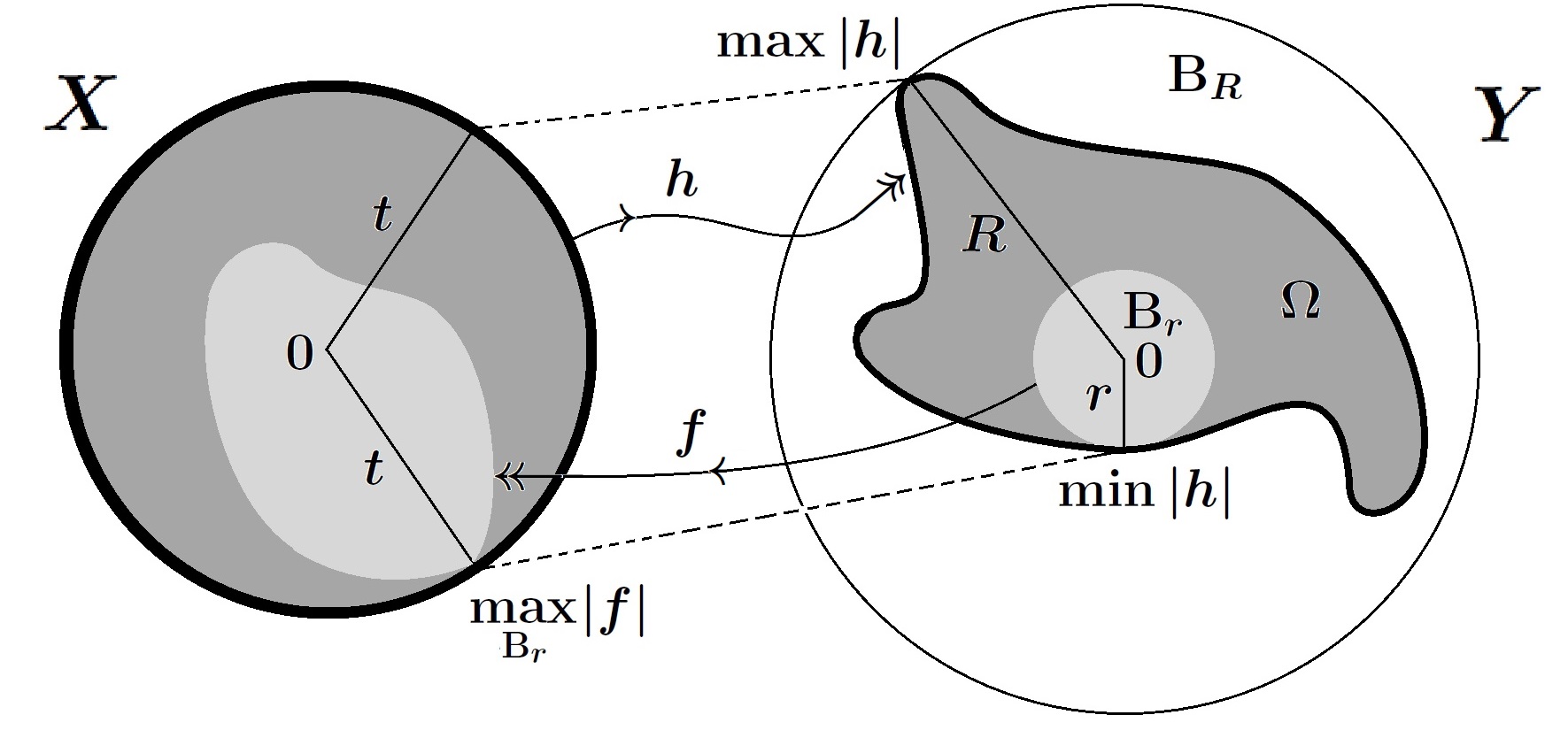}
\caption{The ratio $\,\frac{R}{r} \leqslant \mathscr K\,$}\label{TheRatioRandr}
\end{center}
\end{figure}
Now the inverse map  $\,f :\mathbb Y \onto \mathbb X\,$  takes $\,\Omega\,$ onto $\,\mathbb B(x_\circ , t )\,$.  In particular, it takes the common point of $\,\partial \mathbf B_r\,$ and $\,\partial \Omega\,$ into a point of $\,\partial \mathbb B(x_\circ , t )\,$. This means that
 $$
  t  = \max_{|y| = r} |f(y)|  \bydef \omega_f(r)
 $$
 The proof is completed by invoking the quasi-inverse condition at (\ref{OneSidedCondition}),
 $$
  R = \omega_h(t) = \omega_h(\omega_f(r))  \leqslant \mathscr K r
 $$
\end{proof}
\subsubsection{Doubling Property}
It is worth discussing another special property of quasiconformal mappings in relation to their bi-modulus of continuity. To simplify matters we confine ourselves to  quasiconformal mappings defined on the entire space,  $\, h : \mathbb R^n \onto \mathbb R^n\,$ and its inverse  $\, f : \mathbb R^n \onto \mathbb R^n\,$.  It turns out that at every point $\,x_\circ \in \mathbb R^n\,$ the optimal modulus of continuity $\,\phi(t) \bydef \omega_h(x_\circ; t) \,$, as well as its inverse function $\, \phi^{-1} : [0, \infty) \onto [0, \infty)\,$  have a doubling property. Observe that $\, \phi^{-1}\,$ is not exactly the optimal modulus of continuity of the inverse map $\, f = h^{-1}\,$, the latter is only quasi-inverse to $\, \phi^{-1}\,$. It should be emphasized at this point that doubling property of the modulus of continuity is rather rare, see our representative examples in Section  \ref{Conditions}.

\begin{proposition} Consider all $\,K\,$-quasiconformal mappings $\,h : \mathbb R^n \onto \mathbb R^n\,$. To every $\, \lambda \geqslant 1\,$ there corresponds a constant $\,\mathscr K_\lambda\,$ (actually the one specified in (\ref{Klambda})\,),  and  there is a constant $\,C_\lambda = C_\lambda(n, K)\,$ (independent of $\,h\,$) such that at every point $\,x_\circ \in \mathbb R^n\,$ we have
\begin{equation}\label{DoublingOfInverseModulus}
\omega_h(x_\circ; \lambda\, t ) \leqslant \mathscr K_\lambda \,\,\omega_h(x_\circ; t )\;
\end{equation}
and
\begin{equation}\label{DoublingOfInverseModulus2}
 \omega_h^{-1}(x_\circ; \lambda\, s ) \leqslant C_\lambda \,\,\omega_h^{-1}(x_\circ; s )
\end{equation}
for all $\, 0 \leqslant t  <\infty\,$ and $\,0 \leqslant s < \infty\,$.
\end{proposition}
\begin{proof}  We may again assume that $\,x_\circ = 0\,$ and $\,h(x_\circ) = 0\,$. This simplifies the notation  $\, \omega_h(x_\circ; t )\bydef \omega_h( t )\,$. The proof of the first inequality is immediate from the three points ratio condition in Proposition  \ref{ThrePointsRatio}\,, which gives us exactly the constant $\, \mathscr K_\lambda\,$ from this condition. Indeed, we have
\begin{itemize}
\item $\,\omega_h(\lambda\,t)  = |h(x_1)| \,$, \,for some $\,x_1 \in \mathbb R^n\,$ with $\,|x_1| = \lambda \, t\,$
\item $\, \omega_h(t) \;=\; |h(x_2)|\,$, \,  for some $\,x_2 \in \mathbb R^n\,$ with $\,|x_2| = \, t\,$
\item Hence,  $\,\frac{|x_1|}{|x_2|}  \;\leqslant \,  \lambda\,. $
 \item  Consequently  $\,\frac{|h(x_1)|}{|h(x_2)|}  \;\leqslant \,  \mathscr K_\lambda\,$, which is the desired estimate.
\end{itemize}
\end{proof}
 Clearly, for every $\, y_\circ \in \mathbb R^n\,$ we also have
 \begin{equation}\label{DoublingForf}
 \omega_f(y_\circ; \lambda\, s ) \leqslant \mathscr K_\lambda \,\,\omega_f(y_\circ; s )\;\;\;\; \textnormal{for all }\;\; 0 \leqslant s < \infty\,,
 \end{equation}
 simply by interchanging the roles of $\,h\,$ and $\,f\,$.\\

 We precede the proof of the doubling condition for $\,\omega_h^{-1}\,$, with a quick lemma.
\subsubsection {A quick lemma on doubling condition}  Consider an arbitrary continuously increasing function $\, \phi : [0, \infty) \onto  [0,\infty)\,$
(in our application, $\, \phi(t) = \omega_h(t)\,$).  It is commonly said that $\,\phi\,$ satisfies doubling condition if there is a
constant $\, C_\phi \geqslant 1\,$ such that $\,\phi(2t) \leqslant  C_\phi\, \phi(t)\,$ for all $\, t \geqslant 0\,$. However,
 it is convenient to work with so-called \textit{generalized doubling condition}, which reads as:

 \begin{equation}\label{generalDoubling}
 \phi(\lambda\, t)  \leqslant C_\phi(\lambda)\, \phi(t) \;,\;\;\textnormal{for all}\; t \geqslant 0
 \end{equation}
 where the \textit{$\,\lambda\,$- constant}\; $\, C_\phi(\lambda) \geqslant 1\,$  is obtained by
 iterating the inequality $\,\phi(2\,t) \leqslant  C_\phi\, \phi(t)\,$.\\
Associated with $\,\phi\,$ is its \textit{quasi-inverse function}. This term pertains to any
continuous and strictly increasing
function $\, \psi : [0, \infty) \onto  [0,\infty)\,$ such that

 \begin{equation}\label{QuasiInverse}
  m \,t   \leqslant \psi ( \phi(t) )\; \leqslant M t\;,\;\; \textnormal{for all}\;  t \geqslant 0
 \end{equation}
 where $\, 0 < m \leqslant 1\leqslant M < \infty\,$ are constants.
 In general, $\, \psi\,$ does not satisfy doubling condition, but its inverse  $\,\psi^{-1} : [0, \infty) \onto  [0,\infty)\,$ does.
 \begin{lemma} To every factor $\,\lambda \geqslant 1\,$  there corresponds a generalized doubling constant for $\,\psi^{-1}\,$. For all $\,t \geqslant 0\,$ we have
 \begin{equation}\label{DoublingForInverses}
 \psi^{-1}(\lambda\, t)  \leqslant C_{\psi^{-1}}(\lambda) \,\,  \psi^{-1}(t) \;. \;\;\;   \textnormal{Explicitly} \,\; C_{\psi^{-1}}  (\lambda) \bydef C_\phi\left(M \lambda /m  \right)\,.
 \end{equation}
 \end{lemma}
 \begin{proof} Choose and fix $\, \lambda \geqslant 1\,$.  Inequality (\ref{QuasiInverse}) is equivalent to:

 \begin{equation} \label{AuxiliaryPsiInequality}
  \psi^{-1} (m \,t )   \leqslant  \phi(t) \; \leqslant  \psi^{-1}(M t) \,,\;\; \textnormal{for all}\;  t \geqslant 0\,.
 \end{equation}
 Upon substitution $\, t  \rightsquigarrow   \frac{\lambda t}{m}\,$  in the left hand side, we obtain

  \begin{center}$\psi^{-1} (\lambda \,t )   \leqslant  \phi(\frac{\lambda \,t}{m}) \; =\;  \phi( \frac{M\lambda}{ m} \,\cdot \, \frac{t}{M}) \leqslant   \, C_\phi(\frac{M\lambda}{ m} )\,\cdot  \,\phi(\frac{t}{M})$ \end{center}
  The proof of the lemma is completed by invoking  the right hand side of inequality (\ref{AuxiliaryPsiInequality}) which,
 upon substitution $\, t  \rightsquigarrow   \frac{t}{ M} \,$,    gives us the desired estimate $\, \phi(\frac{t}{M}) \leqslant \psi^{-1}(t)\,$.
  \end{proof}

We summarize this section with the following theorem, which is an expanded version of Theorem~\ref{thm:qc}:

\begin{theorem} Let $\,h : \mathbb R^n \onto \mathbb R^n\,$ be a $\,K\,$-quasiconformal mapping and $\,f : \mathbb R^n \onto \mathbb R^n\,$ its inverse.
Choose and fix an arbitrary point $\,x_\circ \in \mathbb R^n\,$ an its image point $\,y_\circ  = h(x_\circ)\,$. Denote by
$\,\phi(t) = \omega_h(x_\circ; t)\,$ the optimal modulus of continuity of $\,h\,$ at $\,x_\circ\,$ and
by
$\,\psi(s) = \omega_f(y_\circ; s)\,$ the optimal modulus of continuity of $\,f\,$ at $\,y_\circ\,$.
 Then the following statements hold true.
\begin{itemize}
\item[(Q1)] The functions $\,\phi\,$ and $\,\psi\,$   are quasi-inverse to each other.
Precisely, there is a constant $\,\mathscr K = \mathscr K(n, K)\,$ such that
\begin{equation}\label{Q1}
\mathscr K^{-1} \,t\,\leqslant \psi (\phi(t)) \; \leqslant   \mathscr K \,t \;\;\textnormal{ and }\;\; \mathscr K^{-1} \,s\,\leqslant \phi (\psi(s))
\; \leqslant   \mathscr K \,s
\end{equation}
for all $\, t, s \in [0, \infty)\,$.
\item[(Q2)] Both $\,\phi\,$ and $\,\psi\,$ satisfy the general doubling condition; that is, for every $\,\lambda \geqslant 1\,$ there is a
constant $\,\mathscr K_\lambda\,$ such that
\begin{equation}\label{Q2}
\phi(\lambda \,t) \; \leqslant   \mathscr K_\lambda \,\, \phi(t)  \;\;\textnormal{ and }\;\; \psi(\lambda \,s) \,\leqslant \mathscr K_\lambda \,\,\psi (s)
\end{equation}
for all $\, t, s \in [0, \infty)\,$.
\item [(Q3)] As a consequence of Conditions $\,(Q1)\,$ and $\,(Q2)\,$, the inverse functions $\,\psi^{-1}\,$ and $\,\phi^{-1}\,$ also satisfy a
general doubling conditions; namely,
\begin{equation}\label{Q3}
\phi^{-1}(\lambda \,s) \; \leqslant   C_\lambda \,\, \phi^{-1}(s)  \;\;\textnormal{ and }\;\; \psi^{-1}(\lambda \,t) \,\leqslant C_\lambda \,\,\psi^{-1} (s)
\end{equation}
for all $\, t, s \in [0, \infty)\,$, where the constant $\,C_\lambda = \mathscr K_\lambda ( \lambda\, \mathscr K^2)\,$.
\end{itemize}
\end{theorem}

Let us now proceed to more general mappings of bi-conformal energy.
\section{A handy metric in $\,\mathbb R^n\,\simeq\, \mathbb R^{n-1} \times \mathbb R$}

It will be convenient to consider the space $\,\mathbb R^n\,$ as Cartesian product $\,\mathbb R^{n-1} \times  \mathbb R\,$,  with the purpose of using cylindrical coordinates.  Accordingly,
 $$
 \mathbb R^n = \mathbb R^{n-1} \times \mathbb R  \,=\, \big \{ X  = (x,t)  ; \; x = (x_1, ... , x_{n-1}) \in \mathbb R^{n-1}\;\,\textnormal{and}\,\; t \in \mathbb R \,\big\}
 $$
 Hereafter, we  change the  notation of the variables;  the lowercase letter  $\,x\,$ designates a point  $\,(x_1, ... x_{n-1}) \in \mathbb R^{n-1}\,$  while the uppercase letter $\,X = (x,t)\,$ is reserved for points in $\,\mathbb R^n\,$.
 The Euclidean norm of $\,x \in \mathbb R^{n-1}\,$ is denoted by $\, |x| \bydef \sqrt{x_1^2 + \,\cdots\,  + x_{n-1}^2 }  \,$.
The space $\,\mathbb R^{n-1} \times \mathbb R\,$ is furnished with the norm

 $$
 \norm X \norm  \bydef  |x|  +  |t|\,,\;\;\textnormal{for}\;\; X = (x, t) = (x_1, ..., x_{n-1} , t )  \in \mathbb R^{n-1} \times \mathbb R
 $$
In this metric the closed unit ball in $\, \mathbb R^{n-1} \times \mathbb R\,$ becomes the Euclidean double cone
 $$
 \mathcal C \,=\, \{ (x,t) \in \mathbb R^n\; ; \; |x| + |t| \leqslant 1\;\; \} \; = \;\mathcal C_+ \cup \mathcal C_- \;\;
 $$
where we split $\,\mathcal C \,$ into the upper and lower cones:
 $$
 \mathcal C_+ \,=\, \{ (x,t)  ; \; |x| + t \leqslant 1\;,\; t \geqslant 0\; \}\;\;,\;\; \mathcal C_- \,=\, \{ (x,t)  ; \; |x| - t \leqslant 1\;,\; t \leqslant 0\; \}
 $$

\section{The idea of the construction of  $\,H : \mathcal C \onto \mathcal C\,$}
Our construction of a bi-conformal energy map  $\,H : \mathcal C \onto \mathcal C\,$, whose optimal modulus of continuity at the origin coincides with that of the inverse map,  will be carried out in two steps. First we construct a  homeomorphism $\,H  : \mathcal C_+ \onto \mathcal C_+\,$ of finite conformal-energy which equals the identity on $\,\partial \,\mathcal C_+\,$. Its inverse  map $\,F \bydef H^{-1} \,:\,\mathcal C_+ \onto \mathcal C_+\,$ will also have finite conformal-energy. The substance of the matter is that  their optimal moduli of continuity ($\,\omega_{_H}\,$ and $\,\omega_{_F} \,$, respectively)  are inverse to each other; thus generally not equal. In fact $\omega_H$ will be stronger that $\omega_F$.  In the second step we  adopt  the modulus of continuity of $\,F : \mathcal C_+ \onto \mathcal C_+\,$  to an extension of $\,H\,$ to $\mathcal C_-$,  simply by reflecting  $\,F\,$ twice about $\,\mathbb R^{n-1}\,$. Let the reflection $\,\mathfrak r : \mathbb R^n \onto \mathbb R^n\,$ be defined by $\,\mathfrak r(x,t) = (x, -t)\,$.  This gives rise to  a map  $\,\mathfrak r \circ F \circ \mathfrak r\, : \mathcal C_- \onto \mathcal C_-$, which we glue to $\,H  : \mathcal C_+ \onto \mathcal C_+\,$ along the common base  $\,\partial \mathcal C_+ \cap \partial \mathcal C_- \subset \mathbb R^{n-1}\,$.  Precisely,   the desired  homeomorphism $\,H : \mathcal C \onto \mathcal C\,$, still denoted by $\, H\,$,  will be defined by the rule

\begin{equation}\label{DefinitionOfH}
 H  \bydef   \left\{\begin{array}{l}
  H : \mathcal C_+ \onto \mathcal C_+  \\
  \mathfrak r \circ F \circ \mathfrak r\, : \mathcal C_- \onto \mathcal C_- \end{array} \right.
 \end{equation}

Its inverse, also denoted by $\,F : \mathcal C \onto \mathcal C \,$,  is defined analogously  by interchanging the roles of $\,F\,$ and $\, H\,$.

\begin{equation}\label{DefinitionOfF}
 F  \bydef   \left\{\begin{array}{l}
  F : \mathcal C_+ \onto \mathcal C_+  \\
  \mathfrak r \circ H \circ \mathfrak r\, : \mathcal C_- \onto \mathcal C_- \end{array} \right.
 \end{equation}
As a result, the  optimal modulus of continuity of $\,H\,$  will be attained in the upper cone $\,\mathcal C_+\,$, whereas the  optimal modulus of continuity of $\,F\,$ will be attained in the lower cone $\, \mathcal C_-\,$. Clearly, they are the same for the double cone $\mathcal C = \mathcal C_+ \cup \mathcal C_-$, and this is the essence of our construction. \\
Explicit formula for $\,H\,$ can easily be stated, see Definition  \ref{DEFofH} in Section \ref{DefinitionH}. Since $\, H : \mathcal C \onto \mathcal C\,$ and its inverse $\, F \bydef H^{-1}  : \mathcal C \onto \mathcal C \, $ are both equal to the identity on $\,  \partial \mathcal C \,$ we can  extend them to $\,\mathbb R^n\,$ as the identity outside $\,\mathcal C\,$. Whenever it is convenient, we shall speak of $\,H : \mathbb R^n \onto \mathbb R^n\,$  and its inverse $\,F : \mathbb R^n \onto \mathbb R^n\,$ as homeomorphisms of the entire space $\,\mathbb R^n\,$ onto itself.

 \vskip 0.1cm
\begin{figure}[!h]
  \begin{center}
\includegraphics*[height=2.0in]{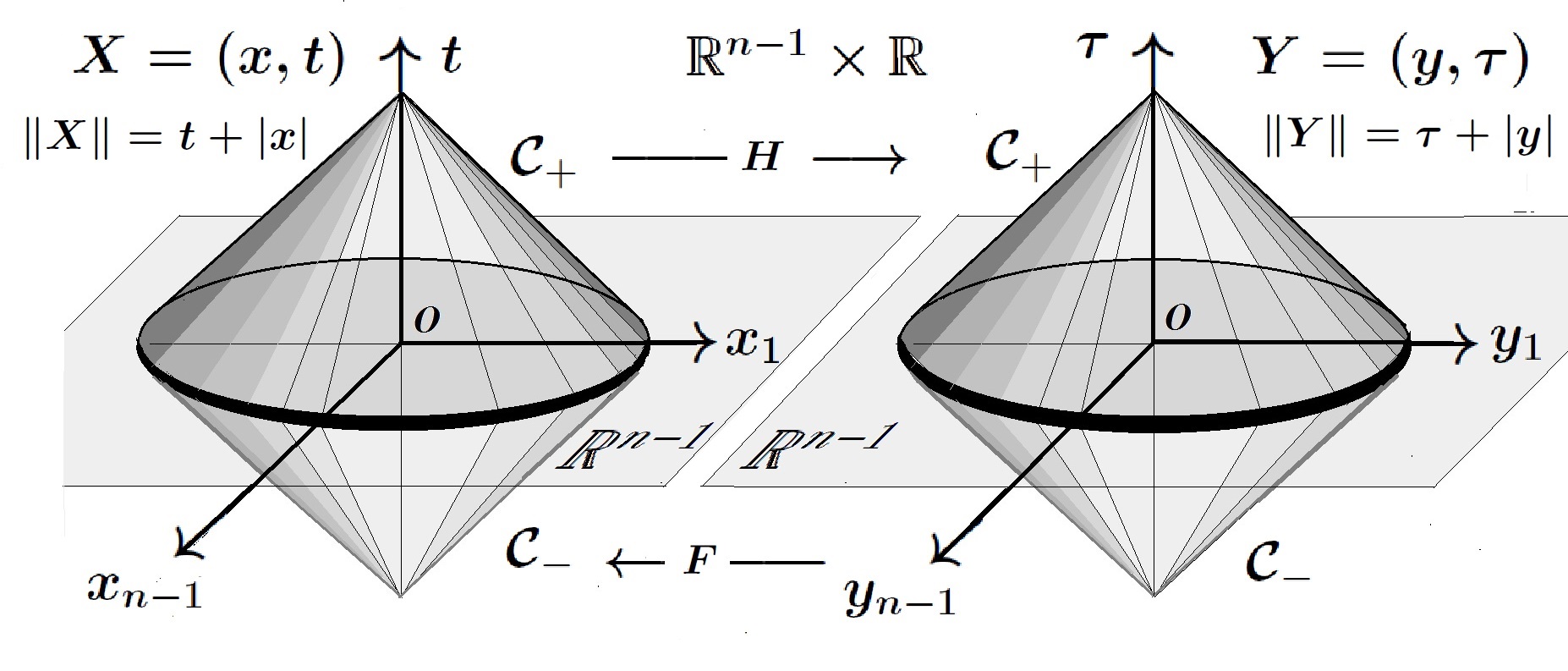}
\caption{A mapping $\,H : \mathcal C \onto \mathcal C\,$ and its inverse $\,F : \mathcal C \onto \mathcal C\,$, will have the same optimal modulus of continuity at the center of $\mathcal C$.}\label{ConformalEnergyMapping}
\end{center}
\end{figure}

\section{Preconditions on the modulus of continuity\\ and the representative examples} \label{Conditions}
 Let us introduce a fairly general class of moduli of continuity to be considered.  These  classes are intended to unify the proofs. It will also give us an aesthetic appearance of the inequalities.  On that account, our moduli of continuity, will be made of functions  $\,\phi : [0, 1]  \onto [ 0, 1]\,$ in $\,\mathscr C[0, 1] \cap \mathscr C^1(0, 1]\,$ such that:
\begin{itemize}
\item [$(\textbf C_1) $]

$\,\phi (0) = 0\,,\; \phi(1) = 1\,$ \; \,( \textnormal{can be  extended by} $\phi(s) = s\,$   \;\textnormal{for}\; $\,s\geqslant 1\,$)
    \item [$(\textbf C_2) $]
   \begin{equation}\label{ConditionC2}
     \phi'(s)  \leqslant \frac{\phi(s)}{s} \leqslant M [\phi'(s)]^2 \;,\; \textnormal{for some contant} \; 1 \leqslant M \,< \infty\,
     \end{equation}
    \item [$(\textbf C_3) $] \textit{Finite Energy Condition} :

\begin{equation}\label{FiniteEnergyCondition}
 E[\phi] \bydef \int_0^1   |\phi(s)| ^n \, \frac{ \textnormal {d} s}{s}\; < \infty\,.
\end{equation}

\end{itemize}

As a consequence of Conditions $\,(\textbf C_1)\,$ and $\,(\textbf C_2)\,$  we have:
\begin{itemize}
\item
\begin{equation}
 \lambda(s) \bydef \frac{\phi(s)}{s}\,\geqslant  \phi'(s) \geqslant \frac{1}{M} \,\,\;\; \textnormal{for all }\,\, \,0 < s \leqslant 1\,\end{equation}

In the forthcoming representative examples (except for $\,\phi(s) \equiv s )\,$ we have even stronger property; namely,  $\, \lim_{s \rightarrow 0} \phi'(s) = \infty\,$ .
\item  The function $\, \lambda(s)\,$ is non-increasing. This follows from
\begin{equation}\label{LambdaPrime}
  \lambda'(s) =  \frac{\phi'(s)}{s} \;-\, \frac{\phi(s)}{s^2}\; \leqslant 0.
  \end{equation}
  \item  Thus in fact,   \begin{equation}\label{GammaBigger1}
  \,\lambda(s)\;=\;  \frac{\phi(s)}{s} \geqslant \;\frac{\phi(1)}{1}  = 1 \,\;, \; \textnormal{for all}\;  \, 0 < s \leqslant 1\,\end{equation}
\end{itemize}

\subsection{Representative examples}
\begin{itemize}
\item [$(\textbf E_0) $] For $\,\;\;0 < \varepsilon \leqslant 1\,$,  we set
\begin{equation}\label{Example0}
\phi_{_0}(s) = s^{\,\varepsilon }\;. \;\;\;\;\textnormal{In the  borderline case,}   \;\;\,\phi(s) = s\,\nonumber
\end{equation}
\item  [$(\textbf E_1) $] For $\, \frac{1}{n} < \alpha \leqslant 1\,$, we set
\begin{equation} \label{Example1}
\phi_{_1}(s)  \,=\, \log^{-\alpha} \left(\frac{e}{s}\right)\;\; = \left(1 + a_1 \log\frac{1}{s} \right) ^{-\alpha}  \nonumber
\end{equation}

\item [$(\textbf E_2) $] For $\, \frac{1}{n} < \alpha \leqslant 1\,$, we set
\begin{equation} \label{Example2}
\phi_{_2}(s) =  \left(1 + a_1 \log\frac{1}{s} \right)^{-\frac{1}{n}}\left(1 + a_2 \log \log \frac{e}{s} \right) ^{-\alpha}\nonumber
\end{equation}

\item [$(\textbf E_3) $]For $\, \frac{1}{n} < \alpha \leqslant 1\,$, we set
\begin{equation} \label{Example3}
\phi_{_3}(s) =  \left(1 + a_1 \log\frac{1}{s} \right)^{-\frac{1}{n}}\left(1 + a_2 \log \log \frac{e}{s} \right) ^{-\frac{1}{n}}   \left(1 + a_3 \log \log\log \frac{e^e}{s} \right) ^{-\alpha}\nonumber
\end{equation}
\end{itemize}
Continuing in this fashion, we define a sequence of functions $\,\phi_{_k}\,,\;  k = 0,1,2, ...\,$ in which the last product-term in the round parantheses involves $\,k\,$-times iterated logarithm and $\,(k-1)\,$-times iterated power of $\,e\,$. All the above functions can be extended by setting $\,\phi_k(s) \equiv s,$ for $\, s \geqslant 1\,$.

\begin{remark}The coefficients  $\,a_k \,$ in the above formulas are adjusted to ensure  the inequality $\, \phi'(s)  \leqslant \frac{\phi(s)}{s} \,$, which  is required by Condition  $(\textbf C_2)\,$. This works well with $\,a_k \bydef  (1 - \frac{1}{n})^{k-1} \,$. Indeed, the reader may wish to verify that the expression $\,\frac{s \phi_k'(s)}{\phi_k(s)} \,$ is increasing, thus assumes its maximum value at $\,s = 1\,$. It is then readily seen that its maximum value is not exceeding $\,\frac{1}{n} \, ( a_1 + a_2 + ... + a_{k-2} ) \; +\; \alpha\, a_k  \;= \; 1 - \big(1 - \alpha \big) \left( 1 - \frac{1}{n} \right)^{k-1}\; \leqslant 1\,$.
\end{remark}
\section{The definition of $\,H : \mathcal C_+ \onto \mathcal C_+\,$\,} \label{DefinitionH}
First we set $\,H\,$ on the vertical axis of the upper cone by the rule.
$$
H(\textbf{0},t)  = (\textbf{0},\phi(t))\, .  
$$
Here and below $ (\textbf{0},t) \bydef  (0,...,0 \,,\, t) \in \mathbb R^{n-1} \times \mathbb R $.
We wish $\,H\,$ to be the identity map on the base of the cone, which consists of points $\,(x, 0) \in \mathbb R^{n-1} \times \mathbb R\,$ with $\,|x| \leqslant 1\,$.
The idea is to connect $\,(x, 0)\,$ with the point $\,(\textbf{0}, |x|\,)\,$ by a straight line segment and map it linearly onto the straightline segment with endpoints at  $\,(x, 0)\,$ and $\,(\textbf{0}, \phi(|x|)\,)\,$. Explicitly,
\begin{definition} \label{DEFofH}
The map $\,H : \mathcal C_+ \onto \mathcal C_+\subset \mathbb R^{n-1} \times \mathbb R\,\,$ is given by the formula
\begin{equation}\label{DefOfHonCplus}
H(x, t)  \bydef  (x, \, t \,\lambda(t + |x|) \,)\; ,\,\textnormal{for}\;  0 \leqslant t \leqslant 1 \;\;\textnormal{and} \; |x| + t \leqslant 1
\end{equation}
where we recall that $\, \lambda (s) \bydef \frac{\phi(s)}{s}\,\; \textnormal{for}\;\; 0 < s \leqslant 1\,.$
\end{definition}
Indeed, for $\, \alpha , \beta \geqslant 0\,$ with $\, \alpha + \beta = 1\,$,   we have
$\,H\big[ \alpha\,(x , 0 ) \;+\; \beta\, (\textbf{0 },|x|) \,\big]   =  \alpha\,(x , 0 ) \;+\; \beta \, (\textbf{0 },\phi(|x|)\,$,
which means that $\,H\,$ is a linear transformation between the above-mentioned segments. A formula for the inverse map $\,F : \mathcal C_+ \onto \mathcal C _+\,$  is not that explicit.
 \begin{figure}[!h]
  \begin{center}
\includegraphics*[height=2.0in]{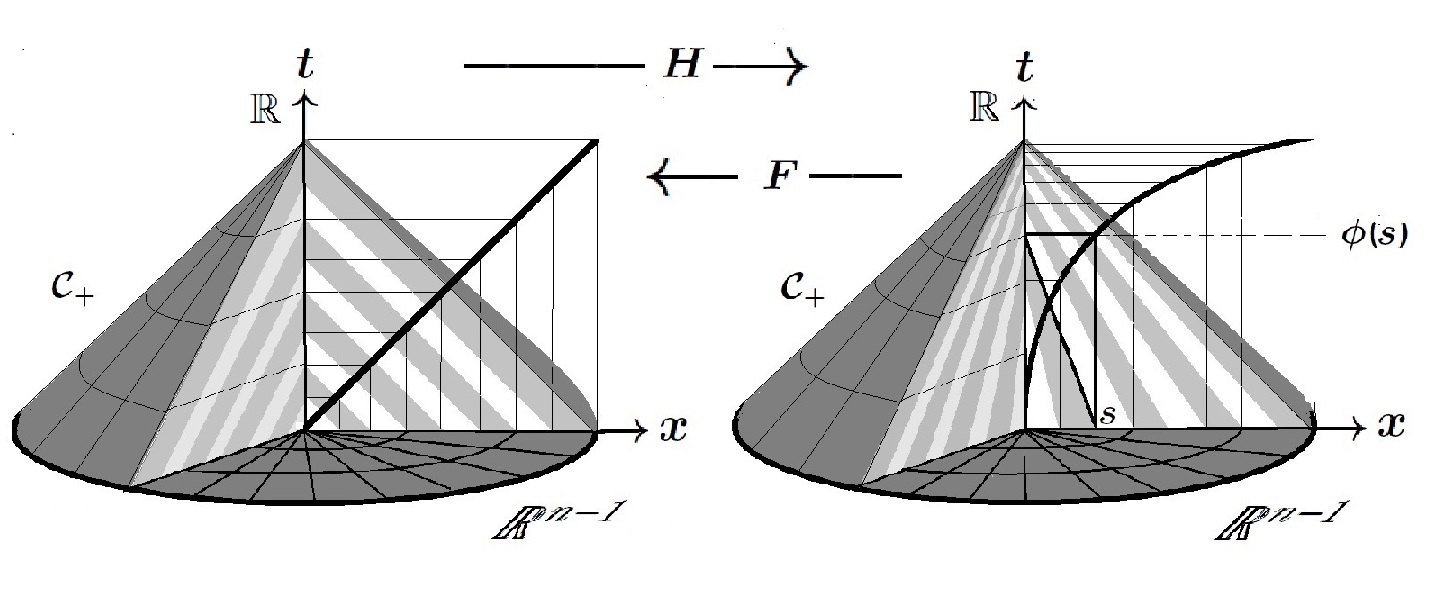}
\caption{Diagonals of rectangles built on the curve $\,t = \phi(s)\,$. The map $\,F\,$ is linear on each such diagonal, as well as on their rotations.  }\label{HandFmappings}
\end{center}
\end{figure}

\section{The Jacobian matrix of $\,H\,$ and its inverse} \label{Jacobian}

A straightforward computation of the Jacobian matrix of $\,H\,$, at the point  $\,X \bydef (x,t) = (x_1, ... , x_{n-1} , t ) \in \mathbb R^{n-1} \times \mathbb R \,$ shows that
\begin{equation}\label{MatrixDH}
DH(x ,\;t)   =
\left( \begin{array}{cccccc}
1 & 0 & 0 \ldots & 0 & 0 \\
$\;$\\
0 & 1 & 0  \ldots & 0 & 0\\
\vdots & \vdots & \vdots  \;\;\;\;\vdots & \vdots & \vdots\\
0 & 0 & 0  \ldots & 1 & 0\\
$\;$\\
\mathfrak D_1 & \mathfrak D_2 & \mathfrak D_3  \ldots & \mathfrak D_{n-1} & \mathfrak D\\
\end{array} \right)
\end{equation}
where  $\, \mathfrak D_i =  \big[\,t \,\lambda ' (t+ |x|)\,\big] \,\frac{x_i}{|x|}\,$ \;and $\, \mathfrak D =  \lambda(t+ |x|)\;  + \; t\, \lambda' (t+ |x|)\,$ is the Jacobian determinant, later also denoted by $\,\mathbf J_H(X)\,$ .
Then the inverse matrix $\,(DH)^{-1}\,$  takes the form

\begin{equation}\label{MatrixDF}
(DH)^{-1}   = \frac{1}{\mathfrak D}\,
\left( \begin{array}{cccccc}
\mathfrak D & 0 & 0 \ldots & 0 & 0 \\
$\;$\\
0 & \mathfrak D & 0  \ldots & 0 & 0\\
\vdots & \vdots & \vdots  \;\;\;\;\vdots & \vdots & \vdots\\
0 & 0 & 0  \ldots & \mathfrak D & 0\\
$\;$ \\
-\mathfrak D_1 & - \mathfrak D_2 & -\mathfrak D_3  \ldots & -\mathfrak D_{n-1} & 1\\
\end{array} \right)
\end{equation}
Square of the Hilbert Schmidt norm  of a matrix is the sum of squares of its entries.  Accordingly,
\begin{equation} \label{NormDH}
| DH |^2 \, = \, n-1 + \big[\,t \,\lambda ' (t+ |x|)\,\big]^2 \, + \big[\,\lambda(t+ |x|)\,  + \, t\, \lambda' (t+ |x|)\,\big]^2
\end{equation}
and
\begin{eqnarray}\label{NormDHinverse}
|\,(DH)^{-1}|^2 \;&=&  \; \frac{ 1 \;} {  \, \big[\,\lambda(t+ |x|)\,  + \, t\, \lambda' (t+ |x|)\,\big]^2} \; + \nonumber\\
& &   \frac{ \big[\,t \,\lambda ' (t+ |x|)\,\big]^2\;} {  \, \big[\,\lambda(t+ |x|)\,  + \, t\, \lambda' (t+ |x|)\,\big]^2} \;\; +\;\; n-1
\end{eqnarray}

\section{The Jacobian determinant $ \mathfrak D  \;=\; \mathbf J_H(X) $ }
We have the following bounds of the Jacobian determinant, including a uniform lower bound for all $\; X = (x, t)\, \in \mathcal C_+\,$.
\begin{equation} \label{LowerBoundJacobian}\, \lambda(\norm X\norm ) \geqslant  \mathbf J_H(X)\;\geqslant  \phi'(\norm X \norm)  \geqslant \frac{1}{M}
 \end{equation}
 \begin{proof} Using the notation $\, \norm X\norm = s = t + |x| \leqslant 1\,$, we write
\begin{eqnarray}\label{DJacobian}
\mathfrak D  \;&=&  \frac{\textnormal{d}}{\textnormal{d} t} \Big( t \, \lambda(t + |x|)  \Big) =  \frac{\textnormal{d}}{\textnormal{d} t} \Big( t \;\frac{\phi(t +|x|)}{t+ |x|}  \Big) = \nonumber\\
& &   \frac{\phi(s)}{s}  \;+ t \left(  \frac{\phi'(s)}{s}\; - \frac{\phi(s)}{s^2}\right) = \frac{\phi(s)}{s} \left( 1 - \frac{t}{s}\right) \; +  t  \, \frac{\phi'(s)}{s} \nonumber \\
\end{eqnarray}

Now, since $\,\phi'(s) \leqslant \frac{\phi(s)}{s}   = \lambda(s)\,$ , it follows that  $\,\mathfrak D \leqslant \lambda(s)\,$. On the other hand $\, \phi'(s) \geqslant \frac{1}{M}\,$ and $\,\frac{\phi(s)}{s} \geqslant 1 \geqslant \frac{1}{M}\,$, whence $\,\mathfrak D \geqslant \frac{1}{M}\,$.
 \end{proof}

 \section{Conformal-energy of $\,H: \mathcal C_+ \onto \mathcal C_+\,$\,}
 In the forthcoming computation the "implied constants" depend only on the dimension $\,n \geqslant 2 \,$.

 \begin{lemma} We have
 \begin{equation}\label{EnergyOfH}
 \int_{\mathcal C_+}| DH(x,t) |^n \,\, \textnormal{d}x\, \textnormal{d}t \; \preccurlyeq  \;E[\phi]
 \end{equation}
 \end{lemma}
 \begin{proof}
 Formula (\ref{NormDH}) yields the inequality:
 \begin{eqnarray}\label{ComputationOfEnergy1}
 \int_{\mathcal C_+}| DH(x,t) |^n \,\, \textnormal{d}x\, \textnormal{d}t  \,&\preccurlyeq &\, 1 + \int_{t + |x| \,\leqslant 1}|\, \lambda(t + |x|)\, |^n \,\, \textnormal{d}x\, \textnormal{d}t  \nonumber \\
  & &  + \int_{t + |x|\, \leqslant 1} t^n | \,\lambda'(t + |x|) \,|^n  \,\, \textnormal{d}x\, \textnormal{d}t
 \end{eqnarray}
 Obviously, for the constant term we have $ 1  \preccurlyeq  \;E[\phi]\,$.
 For the first integral in the right hand side we make the substitution $\, t  =  s  -  |x|\,$ and use Fubini's formula to obtain

 \begin{eqnarray}
  & & \int_{t + |x|}\,| \lambda(t + |x|)\, |^n \, \textnormal{d}x\, \textnormal{d}t = \int_{|x| \leqslant s \leqslant 1}|\, \lambda(s)\, |^n \,\, \textnormal{d}x\, \textnormal{d}s \, =\,  \nonumber \\
  & & = \int_0^1 |\lambda(s)|^n \left(\int_{|x| \leqslant s} \textnormal{d} x \right) \textnormal{d}s  = \frac{\omega_{n-2}}{n-1} \int_0^1  s^{n-1} |\lambda(s)|^n  \textnormal{d}s \, =\, \; \nonumber \\
  & &   =  \frac{\omega_{n-2}}{n-1} \int_0^1  |\phi(s)|^n  \;\frac{\textnormal{d}s}{s} \,\; \preccurlyeq \,\; E[\phi] \;\nonumber
 \end{eqnarray}
 For the second integral in (\ref{ComputationOfEnergy1}),  we make the same substitution $\, t  =  s  -  |x|\,$ and proceed as follows

 \begin{eqnarray}\label{ComputationOfEnergy}
  & & \int_{t + |x| \leqslant 1 } t^n|\, \lambda'(t + |x|)\, |^n \, \textnormal{d}x\, \textnormal{d}t = \int_{|x| \leqslant s \leqslant 1}\left(s - |x|\right)^n |\, \lambda'(s)\, |^n \,\, \textnormal{d}x\, \textnormal{d}s \nonumber \\
   & & = \int_0^1 |\lambda'(s)|^n \left(\int_{|x| \leqslant s} \left(s - |x|\right)^n \textnormal{d} x \right) \textnormal{d}s  \nonumber \\
  & &  =\,c_n\int_0^1 s ^{2n-1} |\lambda'(s)|^n  \textnormal{d}s  \preccurlyeq  \int_0^1  |\phi(s)|^n \,\frac{ \textnormal{d}s }{s}\, =\,  E[\phi] \, \;\nonumber
 \end{eqnarray}
 Here, we used  the inequality $\,|\lambda'(s)\,|  =  \frac{\phi(s)}{s^{\,2}} -  \frac{\phi'(s)}{s}  \leqslant    \frac{\phi(s)}{s^{\,2}}\,$. \\The proof is complete.
\end{proof}

 \section{Conformal-energy of the inverse map}\label{ConEnerF}

 This brings us back to the seminal work \,\cite{AIMO}\, on extremal mappings of finite distortion. Going into this in detail would take us too far afield, so we confine ourselves to a simplified variant.\\
 Consider a homeomorphism $\, H : \mathbb X \onto \mathbb Y\,$ between bounded domains of Sobolev class $\,\mathscr W^{1,n}_{\textnormal{loc}}(\mathbb X, \mathbb Y)\,$ and assume (just to make it easier) that the Jacobian $\,\mathbf J_H \bydef \textnormal{det [DH]}\,$ is positive almost everywhere, as in (\ref{LowerBoundJacobian}).

 \begin{definition}\label{InnerDistortion}
The differential expression
\begin{equation}\label{InnerDist}\,\mathbf K_H (X) \bydef   \big |\,[DH(X)]^{-1}\,\big |^n \; \mathbf J_H(X)  \;=\; \frac{\big|\,D^\sharp H(X)\,\big|^n}{ \; \left [\,\mathbf J_H(X)\,\right]^{n-1} }\,,\end{equation}
is called the \textit{inner distortion} function of $\,H\,$. Here the symbol $\,D^\sharp H\,$ stands for the cofactor matrix of $\,DH\,$, defined by Cramer's rule.
\end{definition}
 The following identity was first observed with a complete proof of it in  \cite{AIMO}\,, see Theorem 9.1 therein.
 \begin{proposition}
 Under the assumptions above, if $\, \mathbf K_H  \in \mathscr L^1(\mathbb X)\,$ then the inverse map $\,F :\mathbb Y \onto \mathbb X\,$ belongs to $\,\mathscr W^{1,n}(\mathbb Y, \mathbb X)\,$  and
 \begin{equation}\label{AnIdentity}
 \int_\mathbb Y \big| \,DF(Y)\,\big|^n\,\textnormal{d} Y \;\;=\;\;   \int_\mathbb X  \mathbf K_H (X)\;\textnormal{d} X\,.
 \end{equation}
 \end{proposition}
In our case, since $\,H\,$ is locally Lipschitz on $\,\mathcal C_+\,$, the derivation of this identity is straightforward. Simply, the differential matrix $\, DF(Y)\,$ at the point $\,Y = H(X)\,$ equals $\,[DH(X)]^{-1}\,$. We may change variables $\,Y = H(X)\,$ in the energy integral for $\,F\,$, to obtain
\begin{equation} \nonumber
\int_{\mathcal C_+} \big |\,DF(Y)\big |^n \; \textnormal{d} Y  \; = \int_{\mathcal C_+} \big |\,[DH(X)]^{-1}\,\big |^n \; \mathbf J_H(X) \,\textnormal{d} X  \; \bydef\, \int_{\mathcal C_+} \mathbf K_H (X)  \,\textnormal{d} X  \;
\end{equation}

Now, by (\ref{MatrixDH}) and (\ref{MatrixDF}), we have a point wise inequality
$ \,\mathbf J_H(X)\; \big|\, [DH(X)]^{-1}\, \big|  \leqslant  \sqrt{n-1}\,  \big|\,DH(X)\,\big|\,$, which yields:

\begin{equation}
\mathbf K_H \leqslant  \frac{(n-1)^{\frac{n}{2}} |\,DH\,|^n}{(\mathbf J_H )^{2n-1}}  \leqslant (n-1)^{\frac{n}{2}} M^{2n-1} |\,DH\,|^n  \, \in \mathscr L^1 (\mathcal C_+)
\end{equation}
because $\,\mathbf J_H(X)  \geqslant \frac{1}{M}\,$, by \,(\ref{LowerBoundJacobian}).
\section{Modulus of continuity of  $\,H: \mathcal C_+ \onto \mathcal C_+\,$  }
We start with the straightforward estimates of the modulus of continuity at $\, ({\bf 0},0) \in \mathbb R^{n-1} \times \mathbb R\,$. In consequence of $\,  \lambda(\norm X\norm) \geqslant 1\,$, we have :

\begin{equation}
\norm X \norm = |x|  + t\, \,\leqslant |x| +  t \,\lambda(t + |x|)  \leqslant \, |x| \,\lambda(\norm X\norm ) +   t \,\lambda(\norm X\norm )\; = \,\phi (\norm X \norm ) \nonumber
\end{equation}
Here the middle term $\,|x| +  t \lambda(t + |x|)  =  \norm\, H(X) \norm \,$. Therefore,
 \begin{equation}\label{OptimalBounds}
 \norm X \norm \leqslant \norm H(X)\norm \;\leqslant \phi(\norm X \norm)
\end{equation}

\begin{corollary} The function $\,\phi\,$ is the optimal modulus of continuity of the map $\, H : \mathcal C_+  \onto \mathcal C_+ \,$  at $\, ({\bf 0}, 0) \in \mathbb R^{n-1} \times \mathbb R_+\,$; that is,
\begin{equation}
 \sup_{\norm X\norm = s} \norm H(X) \norm =  \phi(s) \,,\;\; \textnormal{whenever}\;\;  X \in \mathcal C_+ \,\;\;  \textnormal{and}\; 0 \leqslant s \leqslant 1
\end{equation}
\end{corollary}
Indeed, the supremum is attained at the point $\,X\, = ({\bf 0}, s)\,$ on the vertical axis of the cone $\,\mathcal C_+\,$, because  $\,H({\bf 0},s) = ({\bf 0}, \phi(s)\,)\,$.
\begin{remark}
It is perhaps worth remarking in advance that both inequalities at (\ref{OptimalBounds}) remain valid in terms of the Euclidean norm of $\,\mathbb R^n\,$ as well, where  $\,|X| = |(x,t)| =  \sqrt{|x|^2 + t^2} \leqslant \norm X \norm  \,$. To this end, since $\,\lambda\,$ is decreasing to its minimum value $\,\lambda(1) = 1\,$,   for $\,X \in \mathcal C_+\,$ we can write
$$
|X|^2  \leqslant |x|^2 + t^2 \lambda^2(\norm X \norm )  \;= |H(X)|^2  \leqslant |x|^2\lambda^2(|X|) +  t^2 \lambda^2(|X|) = \phi ^2(|X|),
$$
Let us record this fact as:
\begin{equation}\label{OptimalBoundsEuclideanSetting}
 | X |\leqslant | H(X)|\;\leqslant \phi(| X |)
\end{equation}
For the inverse map $\, F = F(Y)\,$, these inequalities take the form
\begin{equation}\label{OptimalBounds2}
 \psi(|Y|)  \leqslant |F(Y)| \;\leqslant |Y |\, \leqslant \phi(|Y|) \;\;\textnormal{for all}\; Y \in \mathcal C_+\, \; \textnormal{because}\; s \leqslant \phi(s)
\end{equation}
where $\,\psi\,;\, [0,1] \onto [0,1]\,$ denotes the inverse function of $\,\phi\,$. This, however, does not necessarily imply that $\,F\,$ is Lipschitz continuous, as shown by our representative examples.
\end{remark}

We shall now prove that  $\,H\,$ admits $\,\phi\,$ as global modulus of continuity; that is,  everywhere in $\,\mathcal C_+\,$. Precisely, we have
\begin{proposition}
For $\, X = (x,t) \in\mathcal C_+\,$ and $\, X' = (x',t') \in\mathcal C_+\,$ it holds that

\begin{equation}\label{GlobalModulusContinuityH}
 \norm H(X)  - H(X') \norm  \;\leqslant 4\, \phi(\norm X - X' \norm )\;
\end{equation}
Thus, according to~\eqref{GlobalOptimalModulusOfContinuity},
\[\Omega_H (t) \le 4 \phi (t) \, . \]
\end{proposition}.

\begin{proof}
Recall that $\,\norm X\,\norm  \bydef |x| + |t|   \,$  and   $\,H(x,t)  \bydef  ( \,x ,\, t \lambda(\norm X\norm) \,)\,$.
Thus
\begin{equation}
 \norm H(X)  - H(X') \norm  \;\leqslant  |x - x'|\;+\; \; | t \lambda(\norm X\norm) \;-\;  t' \lambda(\norm X'\norm)\,|
\end{equation}

The first term is easily estimated as   $\,|x - x' | \leqslant \phi(|x - x' | ) \leqslant \phi(\norm X - X' \norm )\,$, because $\ s \leqslant \phi(s)\,$ and $\,\phi\,$  is increasing in $\,s \in [ 0, 1 ]\,$. The second term needs more work. First observe that for $\; 0 < A \leqslant B \leqslant 1\,$  it holds:
\begin{equation}\label{LambdaInequality}
  0 < \lambda(A)  -  \lambda(B) \; \leqslant A^{-1}  \phi(B-A)
\end{equation}
 Indeed,
 \[
\begin{split}
 \lambda(A) \;-\;\lambda(B)  &= n \frac{\phi(A) - \phi(B)}{A} \; +\; \frac{B-A}{A} \,\lambda(B) \\ & \leqslant  \frac{B-A}{A} \,\lambda(B - A) =  A^{-1}  \phi(B-A)
 \end{split}
 \]

  In the above formula,   the first term is negative because $\,\phi\,$ is increasing. In the second term we have used the inequality $\,\lambda(B) \leqslant \lambda(B-A)\,$, because $\,\lambda\,$ is nonincreasing.\\
  In   Inequality   (\ref{GlobalModulusContinuityH}) we may (and do)  assume that $\,\norm X'\norm \leqslant \norm X \norm\,$, for otherwise we can interchange $\,X\,$ with $\,X'\,$.  This yields $\, \frac{1}{2}\,\norm X - X' \norm  \leqslant  \norm X\norm\,$ and, consequently, $\,\lambda(\norm X \norm)  \leqslant  \lambda(\frac{1}{2}\norm X - X' \norm )  = \phi( \frac{1}{2}\norm X - X'\norm ) / \frac{1}{2}\norm X - X'\norm  \leqslant 2 \,\phi(\norm X - X'\norm ) /  \norm X - X'\norm \,$ .
Having this and  (\ref{LambdaInequality}) at hand, we conclude with the desired estimate
\[
\begin{split}
|\; t \lambda(\norm X\norm) \;-\;  t' \lambda(\norm X'\norm)\,\,| & \leqslant  |\, t - t'| \;\lambda(\norm X\norm)   +   t'\,|\,\lambda(\norm X\norm) \;-\; \lambda(\norm X'\norm)\,| \\
& \le 
\frac{2\, |\, t - t'| }{\norm X - X'\norm }\; \phi(\norm X - X'\norm  ) \\ & \; \; \; + \frac{t'}{\norm X'\norm}\; \phi (\norm X\norm  - \norm X'\norm ) \\
& \le
 2\,\phi(\norm X - X'\norm  )  \;\; + \; \phi(\norm X - X'\norm  )  \\ & = 3 \, \phi(\norm X - X'\norm  )
\end{split}
 \]
 
\end{proof}

\section{Modulus of continuity of  $\,F: \mathcal C_+ \onto \mathcal C_+\,$  }

All representative functions $\,\phi = \phi_k\,, k = 0, 1 , \dots $ that are listed in  $(\mathbf E _0) \dots (\mathbf E _k) \dots$ are concave ($\, \phi_k'' \leqslant 0\,\,$)  near the origin, but not necessarily in the entire interval $\,[0,1]\,$. Actually, upon minor modifications away from the origin all the above functions can be made concave in the entire interval $\,[0, 1]\,$. But  their aesthetic appearance will be lost. Thus, rather than modifying those examples, in the first step we restrict our attention to a neighborhood of the origin.  Outside such a neighborhood the mapping $\,F : \mathcal C_+ \onto \mathcal C_+\,$ is Lipschitz continuous. This will take care of the global estimate. \\
The additional condition imposed on $\,\phi\,$ reads as follows:
\begin{itemize}
\item [$(\textbf C_4) $]  There is an interval $\, (0, r] \subset (0, 1]\,$  in which $\phi\,$ is $\,\mathcal C^2\,$-smooth and concave; that is,
\begin{equation}\label{SecondDerivativeOfPhi}
 \phi''(s) \;\leqslant  \;0 \;,\;\;\textnormal{for} \;\; 0 < s \leqslant r
\end{equation}
\end{itemize}
We shall now prove that  $\,F\,$ admits $\,\phi\,$ as global modulus of continuity in $\,\mathcal C_+\,$.
\begin{proposition}\label{UniformModulusForF}
For arbitrary two points $\, Y = (y,\tau) \in\mathcal C_+\,$ and $\, Y' = (y',\tau') \in\mathcal C_+\,$ it holds:

\begin{equation}\label{GlobalModulusContinuityF}
 \norm F(Y)  - F  (Y') \norm  \;\;\preccurlyeq\, \;\phi(\norm Y - Y' \norm )\;
\end{equation}
The implied constant depends on the conditions imposed  through  $(\mathbf C_1)-(\mathbf C_4)\,.$
\end{proposition}

\begin{proof}
A seemingly routine proof below, actually took an effort to accomplish all details. Let us begin with the definition of the map $\,H : \mathcal C_+ \onto \mathcal C_+\,$ and some new related notation. For $\,X = (x,t) \in \mathcal C_+ \subset \mathbb R^{n-1} \times \mathbb R\,$, we recall that

$$
\norm X\norm =  |x|\,+\, t \,\;\textnormal{and}\;\;  H(X)  = \big( x , t\,\lambda(t + |x|) \big) \;\bydef \;  \big( y ,  \tau \big) \,=\,Y  \in \mathbb R^{n-1} \times \mathbb R
$$
For the inverse map  $\,F = H^{-1}\,$ we write
$$
\norm Y \norm  \,= |y| \, +  \tau  \,\; \textnormal{and}\;\; F(Y) = \big( y, T \big ) \in \mathcal C_+ \subset \mathbb R^{n-1} \times \mathbb R
$$
where the vertical coordinate $\,T =  T(\tau, |y|) \,$ is determined uniquely from the equation
\begin{equation}\label{EqnForT}\, T\,\lambda ( T  + |y|)\, =\, \tau\,
\end{equation}
In much the same way as in~\eqref{DJacobian} we find that the function $\, T \rightsquigarrow  T\,\lambda ( T  + |y|)\,$ is strictly increasing. We actually have
$$
 \frac{\textnormal{d}\,T\,\lambda ( T  + |y|)}{\textnormal{d}T} = \lambda(T + |y|) \; +\; T \lambda'(T + |y|) \; \geqslant \frac{1}{M}
$$

Even  more can be said about the above expression. Indeed, denoting by $\, s \bydef T + |y| \leqslant 1\,$, we have the identity
$$
\lambda(s)  + T \lambda'(s) \; =\; \frac{\phi(s)}{s}\; + \; T \cdot \left( \frac{\phi'(s)}{s}\,-\, \frac{\phi(s)}{s^2}\,\right)  = \; \left(1 - \frac{T}{s}   \right) \frac{\phi(s)}{s} \; + \;\frac{T}{s}\cdot \phi'(s)
$$
which, in view of Condition ($\mathbf C_2$)\, at (\ref{ConditionC2}), also yields a useful upper bound,

\begin{equation}\label{LowerBound}
\frac{\phi(s)}{s}  \;\geqslant  \lambda(s)  + T \lambda'(s) \; \geqslant \phi'(s)  \geqslant \,\frac{1}{M}
\end{equation}

The latter follows from the Condition ($\mathbf C_2$)\, at (\ref{ConditionC2})\, as well.\\
Now, implicit differentiation in (\ref{EqnForT}) with respect to $\,\tau\,$-variable  gives
\begin{equation}\label{T/tau}
 0 \leqslant \frac{\partial T(\tau, |y|) }{\partial \tau}\;=\; \frac{1}{ \lambda(T + |y|) \; +\; T \lambda'(T + |y|)}  \;\leqslant M
\end{equation}
On the other hand, differentiation with respect to the $\,|y|\,$-variable gives

\begin{equation}\label{T/y}
0 \leqslant \frac{\partial T(\tau, |y|) }{\partial\, |y|}\,=\,\frac{ -\,T \lambda'(T + |y|)}{\lambda(T + |y|) \, +\, T \lambda'(T + |y|)}  \,=\, \frac {- T \lambda'(s)}{ \lambda(s) \,+\, T \lambda'(s)}
\end{equation}
It should be noted that $\,\lambda '(s) \leqslant 0\,$ whenever $\, s \bydef T + |y| \leqslant 1\,$. Precisely,
\begin{equation}
 0 \leqslant \, - \lambda '(s)   \;=\; - \,\frac {\phi'(s)}{s} \; +\; \frac{\phi(s)}{s^2}\; \leqslant  \frac{\phi(s)}{s^2}\,
\end{equation}

From this and the lower bound in (\ref{LowerBound}) we infer that

$$
0 \leqslant \frac{\partial T(\tau, |y|) }{\partial\, |y|}\; \leqslant \; \frac{T\phi(s)}{s^2 \phi'(s)} \; \leqslant \;\frac {\phi(s)}{s\, \phi'(s)}\;\leqslant \,  M \phi'(s)\, = M \phi'(T + |y|),
$$
the latter being guaranteed by the right hand side of inequality (\ref{ConditionC2}).\\

It is at this point that we are going to use the additional assumption that $\,\phi\,$ is concave near the origin; namely, $\,\phi'\,$ is non-increasing in $\,(0, r] \subset (0 , 1 ]\,$. Examine  an arbitrary point $\,Y = (y, \tau) \in \mathcal C_+\,$ of lengths $\,\norm Y \norm  \bydef \,\tau + |y| \leqslant \frac{r}{M}\,$ to show that $\, T + |y| \leqslant r\,$. Recall that $\,T\,$  is determined by the equation $\, T \lambda(T + |y|)  = \tau\,$.  Thus, we have
$ \, \frac{T}{M}  \leqslant \phi'(T+ |y|)\, T  \;\leqslant  \frac{\phi(T+|y|)}{T +|y|}\,T  \;=\;  \tau \, $, by Condition (\ref{ConditionC2}). Hence $\, T + |y| \leqslant M \tau + |y| \leqslant M ( \tau + |y|) \; \leqslant r\,$.
Since $\, s = T + |y| \geqslant |y|\,$ and $\,\phi'\,$ is non-increasing in $\,(0,r]\,$ , we infer that
\begin{equation}\label{EstimateTyDerivative}
0 \leqslant \frac{\partial T(\tau, |y|) }{\partial\, |y|} \; \leqslant M \phi'(|y|)\,\,,\;\;          \textnormal{whenever}  \; \tau + |y| \,\bydef \,\norm Y\norm  \leqslant  \frac{r}{M}
\end{equation}
We are now ready to formulate an estimate of the modulus of continuity of $\,F\,$ within the neighborhood of the origin that is determined by  $\, \norm Y\norm  \leqslant \frac{r}{M}\,$.

\begin{proposition}\label{PropositionOnF}  Let $\,Y = (y, \tau) \in \mathbb R^{n-1}\times \mathbb R\, $ and $\,Y' = (y', \tau') \in \mathbb R^{n-1} \times \mathbb R\, $ be points in $\,\mathcal C_+\,$ such that $\,\norm Y\norm \leqslant \frac{r}{M}\,$ and $\,\norm Y'\norm \leqslant \frac{r}{M}\,$. Then
\begin{equation}
\norm F(Y)\;-\; F(Y')\,\norm \; \leqslant \; 3 M  \phi(\norm Y -  Y'\norm )
\end{equation}
\end{proposition}
\begin{proof} With the notation for $\,F(Y) = ( y\,, \, T(\tau, |y|) )\, $ and $\,F(Y') = ( y'\,,\, T(\tau', |y'|)) \, $ we begin with the computation,

$$
\norm F(Y)  - F( Y')\,\norm = |y - y'| \; + \;|\, T(\tau, |y|) -  T(\tau', |y'|)\,|\; \leqslant
$$
$$
 |y - y'|\; +\; |\,T(\tau, |y|) -  T(\tau, |y'|)\,| \;+\;  |\,T(\tau, |y'|) -  T(\tau', |y'|)\,|\leqslant  (\textnormal{in view of} \;(\ref{T/tau}))
$$
$$
\leqslant \,|y - y'|\; +\; |\,T(\tau, |y|) -  T(\tau, |y'|)\,| \;+\; M\,|\tau - \tau'| \;\leqslant
$$
$$
|\,T(\tau, |y|) -  T(\tau, |y'|)\,| \,+\, M\,\norm Y - Y' \norm   \leqslant  |\,T(\tau, |y|) -  T(\tau, |y'|)\,| \, +\, M\,  \phi(\norm Y - Y' \norm)
$$
The latter is obtained  by the inequality $\, s \leqslant \phi(s) \,$, see (\ref{GammaBigger1}).  It remains to establish the following estimates, say  when $\,  0 < |y'| \leqslant |y| \leqslant r\,. $
\begin{equation} \label{DifferenceOFT}
 |\,T(\tau, |y|) -  T(\tau, |y'|)\,|  \; \leqslant  2 M \phi(|y -y'| ) \; \leqslant \,2 M \,\phi(\norm Y - Y'\norm )
\end{equation}
  To that end, we begin with the following expression:
  $$
  T(\tau, |y|) -  T(\tau, |y'|) \; = \int_0^1 \frac{\textnormal d}{\textnormal d \gamma } \; \Big[ T(\tau, |\gamma \,y + (1-\gamma) y'|)\Big]\; \textnormal{d}\gamma
  $$
$$
 =  \int_0^1  \;  T_\xi(\tau, |\gamma \,y + (1-\gamma) y'|)\, \left \langle\frac{\gamma \,y + (1-\gamma) y'\,}{|\gamma \,y + (1-\gamma) y'| }\;\Big |\; y - y'\,\right \rangle\, \textnormal{d}\gamma
$$
where $\,T_\xi (\tau, \xi) \,\bydef \frac{\partial T(\tau, \xi) }{\partial\, \xi}\,$ . In view of (\ref{EstimateTyDerivative})\,, we obtain
\begin{equation}\label{DifferenceOfT}
|\;T(\tau, |y|) -  T(\tau, |y'|)\;| \leqslant\, M |y- y'| \, \int_0^1  \; \phi'(|\gamma \,y + (1-\gamma) y'|)\; \textnormal{d}\gamma
\end{equation}
It is important to notice that $\,|\gamma \,y + (1-\gamma) y'| \leqslant r\,$ , which enables us to invoke Condition $\,(\textbf C_4)\,$  at (\ref{SecondDerivativeOfPhi}); that is, $\,\phi'\,$ is non-increasing  in the interval $\,(0, r]\,$. The following interesting lemma comes into play.

\begin{lemma}  Let $\,\Phi : (0 , r] \rightarrow (0, \infty)\,$ be continuous non-increasing and integrable,
$$
 \int_0^r \Phi(s) \,\textnormal{d}s \; <\; \infty\;.
$$
Then for every  vectors $\, \mathfrak{a} , \mathfrak{b} \,$ in a normed space  $\, ( \mathfrak N \, ; \,|\,.\,| )\,$, such that \; $\, 0 < | \mathfrak a | \leqslant r \,$  and  $\, 0 < | \mathfrak b | \leqslant r \,$, it holds:
\begin{equation}\label{LemmaonPhi}
\int_0^1 \Phi (| \gamma \mathfrak a  \; +\; (1-\gamma) \,\mathfrak b\,| )\, \textnormal{d} \gamma  \; \leqslant \;\frac{1}{|\mathfrak a |  + |\mathfrak b|} \left( \int_0^{|\frak a|} \Phi(s)\,\textnormal{d} s\; +\;  \int_0^{|\frak b|} \Phi(s)\,\textnormal{d} s \right)
\end{equation}
Equality occurs if  $\, \mathfrak a\,$ is a negative multiple of  $\,\mathfrak b\,$.
\end{lemma}

\begin{proof}
Since $\,\phi\,$ is non-increasing, by triangle inequality it follows that
$$
\int_0^1 \Phi (| \gamma \mathfrak a  \; +\; (1-\gamma) \,\mathfrak b\,| )\, \textnormal{d} \gamma  \; \leqslant \; \int_0^1 \Phi \left(\big|\, (1-\gamma) \, |\mathfrak b |\, -\, \gamma |\mathfrak a| \;\big| \right)\, \textnormal{d} \gamma  \; =
$$
$$
\int_0^{\frac{|\mathfrak b|}{|\mathfrak a| + |\mathfrak b}|} \Phi \big(\, (1-\gamma) \, |\mathfrak b |\, -\, \gamma |\mathfrak a| \; \big)\, \textnormal{d} \gamma     \;+\; \int_{\frac{|\mathfrak b|}{|\mathfrak a| + |\mathfrak b}|}^1 \Phi \big(\,\gamma |\mathfrak a|\,-\, (1-\gamma) \, |\mathfrak b |\, \big)\, \textnormal{d} \gamma
$$
 In the first integral we make a substitution  $\, s =  (1-\gamma) \, |\mathfrak b |\, -\, \gamma |\mathfrak a|\,$, which places $\, s\, $ in the interval $\, (0 , |\mathfrak b|\, )\,$ and $\, |\textnormal d s |  = ( |\mathfrak a| +|\mathfrak b|) \,\textnormal d \lambda\,$. This gives us the second integral-term of the right hand side of (\ref{LemmaonPhi}), and similarly for the first integral-term.
\end{proof}
Since $\,\phi'\,$ is non-increasing in the interval $\,(0, r ]\,$ (by inequality (\ref{SecondDerivativeOfPhi})  at Condition \,($\mathbf C_4$) ), we may apply Estimate (\ref{LemmaonPhi})\, to $\, \Phi = \phi'\,$.
Now, returning to (\ref{DifferenceOfT}), the inequality \,(\ref{DifferenceOFT})\, is readily inferred as follows:
$$
|\;T(\tau, |y|) -  T(\tau, |y'|)\;| \leqslant\, M |y- y'| \,\frac{\phi(|y|)  + \phi(|y'|)}{|y|  + |y'|} =
$$
$$
\,=\; M \, \phi(|y - y'|) \, \frac{|y-y'|}{\phi(|y-y'|)}\,\frac{\phi(|y|)  + \phi(|y'|)}{|y|  + |y'|}\; \leqslant
$$
$$
\, \leqslant M \, \phi(|y - y'|) \, \frac{|y| +|y'|}{\phi(|y|+|y'|)}\,\frac{\phi(|y|)  + \phi(|y'|)}{|y|  + |y'|}\; \leqslant 2 \, M \, \phi(|y - y'|)\,,
$$
because $\, \frac{s}{\phi(s)} = \frac{1}{\lambda(s)}\,$ \,is non-decreasing, see \,(\ref{LambdaPrime})\,, and $\,\phi(s)\,$ is increasing.
The proof of  Proposition \ref{PropositionOnF}\, is complete.
\end{proof}

 Finally, the global estimate (\ref{GlobalModulusContinuityF})  in Proposition \ref{UniformModulusForF}  follows from Proposition \ref{PropositionOnF},   whenever  $\,\norm Y\norm \leqslant \frac{r}{M}\,$ and $\,\norm Y'\norm \leqslant \frac{r}{M}\,$. Whereas its extension to all points $\,Y\,$ and $\,Y'\,$  is  fairly straightforward by invoking Lipschitz continuity of $\, F\,$  away from the origin.
\end{proof}

\section{Conclusion} \label{MainResult}
Choose  an arbitrary modulus of continuity function $\,\phi : [0, \infty)  \onto  [0, \infty)\,$ that satisfies conditions $\,(\mathbf C_1) \, (\mathbf C_2) \, (\mathbf C_3)\,$ and $\,(\mathbf C_4)\,$. Then consider a bi-conformal energy map $\, H :\mathcal C_+ \onto \mathcal C_+\,$ defined in \,(\ref{DefOfHonCplus})  together with its  inverse map $\, F :\mathcal C_+ \onto \mathcal C_+\,$. Extend $\, H\,$ and $\, F\,$  to the double cone $\,\mathcal C = \mathcal C_+ \cup \mathcal C_-\,$ by the reflection rule at\, (\ref{DefinitionOfH}). Afterwards extend $\, H\,$ and $\, F\,$ to the entire space $\,\mathbb R^n\,$  by setting $\, H = \textnormal{Id} : \mathbb R^n \setminus \mathcal C \onto \mathbb R^n \setminus \mathcal C\,$  and $\, F = \textnormal{Id} : \mathbb R^n \setminus \mathcal C \onto \mathbb R^n \setminus \mathcal C\,$. Then we obtain:

  \begin{figure}[!h]
\begin{center}
\includegraphics*[height=2.10in]{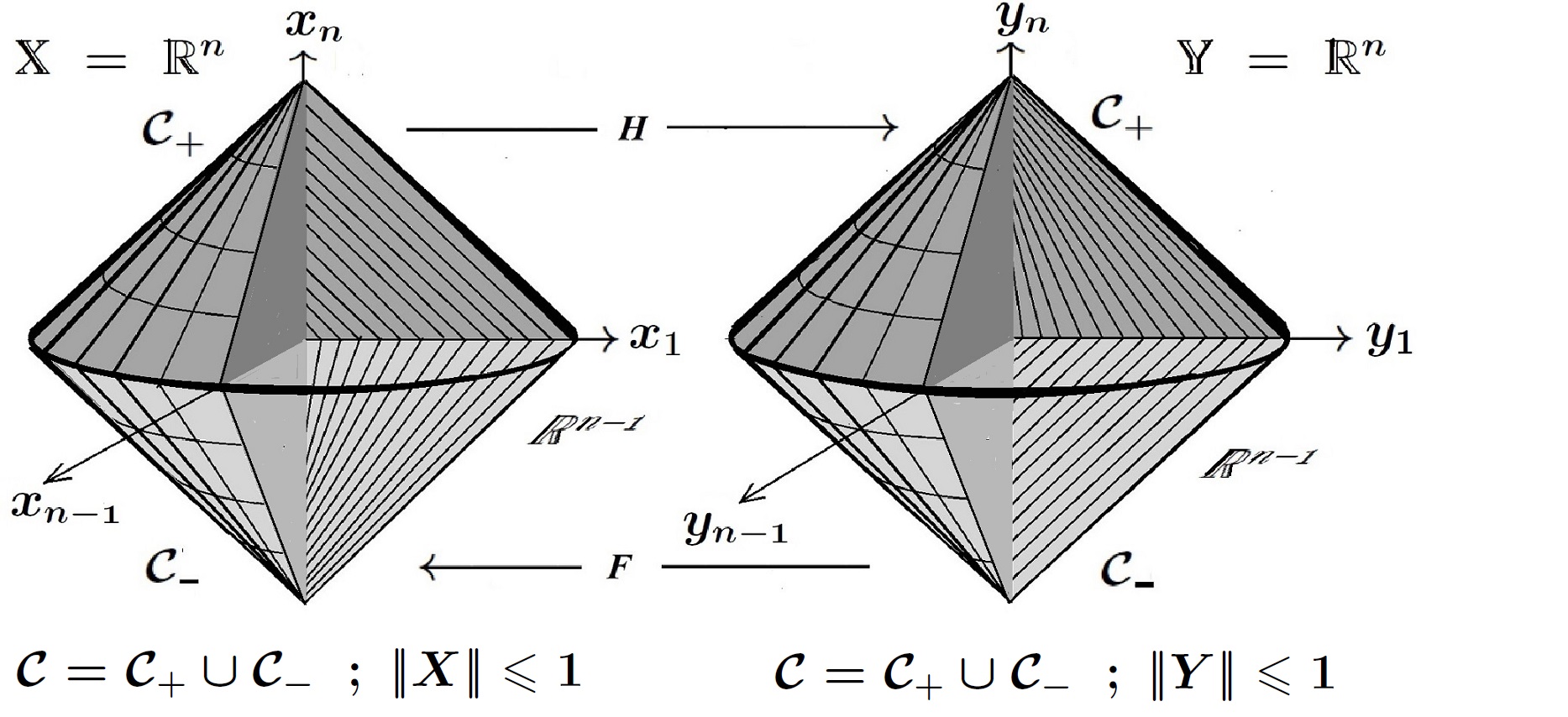}
\caption{Bi-conformal energy mapping $\,H\,$ and its inverse $\,F\,$ exhibit the same optimal modulus of continuity at the origin of the double cone $\mathcal C$. }\label{BiConformalEnergyMapping}
\end{center}
\end{figure}

\begin{theorem}\label{MainTheorem} For every modulus of continuity function $\,\phi : [0, \infty)  \onto  [0, \infty)\,$ satisfying conditions $\,(\mathbf C_1) \, (\mathbf C_2) \, (\mathbf C_3)\,$ and $\,(\mathbf C_4)\,$, there exists a homeomorphism $\,H : \mathbb R^n \onto \mathbb R^n\,$ of Sobolev class $\,\mathscr W^{1,n}_{\textnormal{loc}} (\mathbb R^n \,,\,\mathbb R^n)\,$, whose inverse $\,F = H^{-1}  : \mathbb R^n \onto \mathbb R^n\,$ also lies in the Sobolev space $\,\mathscr W^{1,n}_{\textnormal{loc}} (\mathbb R^n , \mathbb R^n)\,.$ Moreover

 \begin{itemize}
 \item  $\,H(0) = 0\,, \; H(X) \equiv X \,,\;\textnormal{for}\; \norm X\norm\geqslant  1 \,\, \textnormal{and for}\, X = (x_1, ... , x_{n-1} , 0 )\,. $
 \item  $\,H   : \mathbb R^n \onto \mathbb R^n\,$  admits $\,\phi\,$ as its global modulus of continuity; that is,
 \begin{equation}\norm\,H(X_1) \,-\, H(X_2)\,\norm  \; \, \preccurlyeq\,  \; \phi(\norm X_1 - X_2\norm)\;,\;\;\textnormal {for all}\;\; X_1 , X_2  \in \mathbb R^n\,.
 \end{equation}
 \item  The inverse map $\,F   : \mathbb R^n \onto \mathbb R^n\,$  satisfies the same condition:
 \begin{equation}\norm\,F(Y_1) \,-\, F(Y_2)\,\norm  \; \, \preccurlyeq\,  \; \phi(\norm Y_1 - Y_2\norm) \;,\;\;\textnormal {for all }  Y_1 , Y_2 \in \mathbb R^n\,.
 \end{equation}

 \item  $H$ and $F$ share the same optimal moduli of continuity at the origin; namely
 \begin{equation}\label{OptimalModulus}
 \omega_{_H}(0,r)  = \max_{\norm X \norm = r } |H(X)| \;=\;\phi(r) \;=\; \max_{\norm Y\norm = r} |F(Y)| = \omega_{_F}(0, r)
 \end{equation}
for all $0 \le r < \infty$.
\end{itemize}
 \end{theorem}

\end{document}